\documentclass[11pt]{amsart}

\usepackage[top=1.12in, bottom=1.12in, left=1.1in, right=1.1in]{geometry}
\allowdisplaybreaks
\usepackage{graphicx}
\usepackage[dvipsnames]{xcolor}
\usepackage{amsfonts,amsmath,amssymb,amsthm,bbm}
\usepackage{enumitem}
\usepackage[normalem]{ulem}
\usepackage{hyperref}
\usepackage{mathtools}
\usepackage{multirow}
\usepackage{comment}

\newtheorem{theorem}{Theorem}
\newtheorem{proposition}[theorem]{Proposition}
\newtheorem{lemma}[theorem]{Lemma}
\newtheorem{corollary}[theorem]{Corollary}

\theoremstyle{remark}
\newtheorem{remark}[theorem]{Remark}

\theoremstyle{definition}
\newtheorem{defn}{Definition}

\theoremstyle{plain} 
\newcommand{\thistheoremname}{}
\newtheorem*{genericthm*}{\thistheoremname}
\newenvironment{namedthm*}[1]
  {\renewcommand{\thistheoremname}{#1}%
   \begin{genericthm*}}
  {\end{genericthm*}}

\newcommand{\N}{\mathbb N}     
\newcommand{\R}{\mathbb R}     
\newcommand{\Z}{\mathbb Z}     

\newcommand{\cal}{\mathcal}
\renewcommand{\epsilon}{\varepsilon}
\newcommand{\del}{\partial}

\renewcommand{\kappa}{\varkappa}
\renewcommand{\l}{\lambda}
\newcommand{\e}{\varepsilon}
\renewcommand{\d}{\delta}

\newcommand{\s}{\sigma}

\newcommand{\fl}[1]{\lfloor #1 \rfloor}  
\newcommand{\flo}[1]{\left\lfloor #1 \right\rfloor} 
\newcommand{\ceil}[1]{\lceil #1\rceil} 
\renewcommand{\t}[1]{\widetilde{#1}} 
\renewcommand{\c}[1]{\check{#1}}
\newcommand{\ind}[1]{ \mathbbm{1}_{\{ #1 \}} }

\newcommand{\red}[1]{{\color{red}#1}}

\definecolor{vio}{rgb}{0.54, 0.17, 0.89}

\let\emptyset\varnothing
\DeclareMathOperator{\gap}{gap}

\title[Convergence of TSAW to TSRM]{Convergence of rescaled ``true'' self-avoiding walks to the T\'oth-Werner ``true'' self-repelling motion}

\author{Elena Kosygina}
\address{Elena Kosygina\\One Bernard Baruch Way \\ Department of Mathematics, Baruch College \\ New York, NY 10010, USA \\and Mathematics Department, NYU Shanghai\\ 567 West Yangsi Rd\\ Pudong, Shanghai, China}
\email{elena.kosygina@baruch.cuny.edu \& ek9@nyu.edu}
\urladdr{http://www.baruch.cuny.edu/math/elenak/}

\author{Jonathon Peterson}
\address{Jonathon Peterson\\Purdue University\\Department of Mathematics\\150 N University Street\\West Lafayette, IN  47907\\USA}
\email{peterson@purdue.edu}
\urladdr{http://www.math.purdue.edu/$\sim$peterson}

\date{\today}



\begin{document}

\begin{abstract}
  We prove that the rescaled ``true'' self-avoiding walk
  $(n^{-2/3}X_{\fl{nt}})_{t\in\R_+}$ converges weakly as $n$ goes to
  infinity to the ``true'' self-repelling motion constructed by T\'oth
  and Werner \cite{TW98}. The proof features a joint generalized
  Ray-Knight theorem for the rescaled local times processes and their
  merge and absorption points as the main tool for showing both the
  tightness and convergence of the finite dimensional
  distributions. Thus, our result can be seen as an example of
  establishing a functional limit theorem for a family of processes by
  inverting the joint generalized Ray-Knight theorem.
\end{abstract}

\maketitle

\section{Introduction and the main result}

The ``true'' self-avoiding walk (TSAW) on $\Z$ is a stochastic process
$(X_n)_{n\in\Z_+}$ which starts from the origin, $P(X_0=0)=1$, and at
each time step moves to one of the two nearest neighbors with
probabilities given by: $P(X_1=\pm 1)=1/2$ and for $n\in\N$
\begin{align}
 P(X_{n+1}&=X_n+1\mid X_0,X_1,\ldots,X_n)=1-P(X_{n+1}=X_n-1\mid X_0,X_1,\ldots,X_n)\nonumber\\ &=\frac{\exp(-\beta ({\cal E}^+_n(X_n)+{\cal E}^-_n(X_n+1)))}{\exp(-\beta ({\cal E}^+_n(X_n)+{\cal E}^-_n(X_n+1)))+\exp(-\beta ({\cal E}^+_n(X_n-1)+{\cal E}^-_n(X_n)))}, \label{tsaw}
\end{align}
where $\beta>0$ is a fixed parameter  and
\begin{equation}\label{Elt}
  \mathcal{E}^+_n(k) := \sum_{i=0}^{n-1} \ind{X_i=k, \, X_{i+1} = k+1},
  \quad 
  \mathcal{E}^-_n(k) := \sum_{i=0}^{n-1} \ind{X_i=k, \, X_{i+1} = k-1},\quad k\in\Z,
\end{equation}
are respectively the numbers of crossings of the directed edges
$k\to k+1$ and $k\to k-1$ before time $n\in\N$. In words, at time $n+1$
the walk makes a step from $X_n$ to $X_n\pm 1$ with probability
proportional to $\lambda:=e^{-\beta}\in(0,1)$ raised to the power
equal to the total number of crossings by time $n$ of the {\em undirected}
edge $\{X_n,X_n\pm 1\}$.

A quantity which is closely related to the edge local times defined
above is the local time of the walk at $k\in \Z$ up to time $n\geq 0$
inclusively,
\begin{equation}
  \label{Lt}
  L(n,k) := \sum_{i=0}^n \ind{X_i=k}=\mathcal{E}^+_n(k)+\mathcal{E}^-_n(k)+\ind{X_n = k}. 
\end{equation}

The TSAW (with site repulsion) on $\Z^d$ was first introduced and
studied in the physics literature, \cite{APP83}, where its properties,
in particular the upper critical dimension, were contrasted with those
of a self-repelling polymer chain.  B.\,T\'oth in \cite{tTSAW} modified
the original model by replacing site repulsion with bond repulsion as
in \eqref{tsaw}. This phenomenologically minor change opened a door to
a rigorous mathematical treatment of TSAW on $\Z$, see also
\cite{TV11} for a continuous time version of a generalization of the
TSAW. The main results of \cite{tTSAW} include
\begin{itemize}
\item [--] a generalized Ray-Knight theorem for local time processes
  stopped at inverse local times (\cite[Theorem 1]{tTSAW}, recalled as
  Theorem~\ref{thm:marginalRK} below) and, as a consequence,
  weak convergence of the inverse local times (\cite[Corollary 2]{tTSAW},
  see Corollary~\ref{cor:tau-lim} with $N=1$);
\item [--] a local limit theorem for the joint law of TSAW and its
  local time at independent geometrically distributed times
  (\cite[(6.6)]{tTSAW}) which immediately implies joint weak
  convergence of TSAW and its local time at independent geometrically
  distributed times, see Theorem~\ref{th:JWLGST} for $k=1$.
\end{itemize}
The results of \cite{tTSAW} naturally led to the conjecture that
rescaled processes $(n^{-2/3}X_{\fl{tn}})_{t\ge 0}$ should converge to
a limit as $n\to\infty$. A candidate $(\mathfrak{X}(t))_{t\ge 0}$ for
(a scalar multiple of) this limit was constructed and studied in a
seminal work \cite{TW98}. The process
$\mathfrak{X}(\cdot):=(\mathfrak{X}(t))_{t\ge 0}$ is known as a
``true'' self-repelling motion (TSRM) (see
Section~\ref{Prelim}). Explicit formulas for the densities of
$\mathfrak{X}(t)$ and of its local time at the current location,
$\mathfrak{H}(t):=\mathfrak{L}(t,\mathfrak{X}(t))$, were given in
\cite{DT13}. Large deviation estimates, the law of iterated logarithm,
and further path properties of the TSRM were obtained in \cite{Dum18}.
In \cite[Section 11]{TW98}, the authors also introduced a discrete
approximation of the pair of processes
$(\mathfrak{X}(\cdot),\mathfrak{H}(\cdot))$. Their approximation
consists of a sequence of lattice-filling paths in $\Z\times\Z_+$
(space\,$\times$\,local time) which, after the diffusive scaling (by
$n$ and $\sqrt{n}$ respectively) should converge weakly to a
plane-filling curve $(\mathfrak{X}(\cdot),\mathfrak{H}(\cdot))$ in
$\R\times\R_+$. The convergence result for this model was
established in \cite{NR06}. We note that this discrete approximation
of the TSRM is very different from the TSAW. In particular, the
local time profiles in this model can only change by $\pm1$
  over each unit of space while, as we shall see below, the
increments of the corresponding processes for TSAW have an
unbounded range.

Hence, somewhat surprisingly, the weak convergence of
$n^{-2/3}X_{\fl{tn}}$ as $n\to\infty$ to a multiple of
$\mathfrak{X}(t)$ at the process level or even for a fixed
deterministic time $t>0$ have not yet been established.  One of the
purposes of our paper is to fill this gap. As an important first step
towards this goal we prove a joint generalized Ray-Knight theorem
(GRKT) for the rescaled local time processes and their merge and
absorption points (Theorem~\ref{thm:JRK-TSAW}). In the next and
probably even more important step, we show how the joint weak
convergence of the rescaled TSAW and its local time at the process
level, in particular tightness (Proposition~\ref{prop:XHtight}), can
be derived {\em directly} from the joint GRKT and properties of
$(\mathfrak{X}(\cdot),\mathfrak{H}(\cdot))$.  This paves the way for a
general method of constructing scaling limits of other
self-interacting random walks for which Ray-Knight theorems have been
proven (see e.g., \cite{tTSAWGBR,tGRK}). For the TSAW considered in
this paper, our task is simplified by the
fact that the limit process has already been constructed and
the properties of the TSRM and its local time
process are well known.

The main result of this paper is the following theorem. Recall that
$\lambda=e^{-\beta}\in(0,1)$ is a fixed model parameter. Throughout
the paper we shall denote the Skorokhod spaces $D([0,\infty))$
  and $D(\R)$ by ${\cal D}_+$ and ${\cal D}$
  respectively. For the reader's convenience the necessary properties
  of the limiting processes are summarized in Section~\ref{Prelim}.

\begin{theorem}\label{thm:TSAW-FLT}
  Let $(\mathfrak{X}(\cdot),\mathfrak{H}(\cdot))$
  be respectively the TSRM and its local time at the current location
  constructed in \cite[(3.15) and Proposition 3.5]{TW98}. Define
\begin{equation}
   \label{sigma}
   \sigma^2 = \frac{\sum_{x \in \Z} x^2 \lambda^{x^2}}{\sum_{x \in \Z}\lambda^{x^2}}.
 \end{equation}
  Then, as
  $n\to\infty$, the sequence of processes
  \begin{equation}\label{discreteTSRM} \left(\frac{X_{\fl{\boldsymbol{\cdot}\,n }}}{(2\sigma)^{-2/3}n^{2/3}},\frac{L\left(\fl{\cdot\, n}, X_{\fl{\boldsymbol{\cdot}\, n}}\right)}{(2\sigma)^{2/3} n^{1/3}}\right)_{n\in\N}
  \end{equation}
converges weakly in the standard Skorokhod topology on
${\cal D}^2_+$ to
 $(\mathfrak{X}(\cdot),\mathfrak{H}(\cdot))$.
\end{theorem}
\begin{remark}
  The constant $\sigma^2$ in \eqref{sigma} is the variance of the
  stationary distribution $\pi(x)\sim\lambda^{x^2}$, $x\in\Z$, of a
  Markov chain on a single site $k\in\N$ whose $n$-th term is equal to
  the excess of the number of right steps by TSAW from $k$ over the
  number of left steps from $k$ at the time when the latter is equal
  to $n$. This Markov chain and the underlying generalized Polya's urn
  processes for the TSAW were introduced and thoroughly studied in
  \cite{tTSAW}, and further analysis of these processes
 (Appendices~\ref{sec:aux} and \ref{sec:overshoot}) yields many
  of the technical results on the TSAW that are needed for our
  work. For the reader's convenience, we briefly review this framework
  in Appendix~\ref{sec:urn}.
\end{remark}

Apart from the mentioned above models studied by B.\,T\'oth, there are
a number of non-Markovian nearest neighbor random walks on $\Z$ for
which edge or site local time processes are Markovian. These local
time processes and, in particular GRKTs, i.e., diffusion
approximations of local times, have been extensively used in the past
to study recurrence properties, zero-one laws, laws of large numbers,
occupation times, scaling limits, and large deviations of such walks,
see, for example, \cite{
  tLTWRRW, TV08,kmLLCRW,
  HTV,
  pERWLDP,kzERWsurvey,mpvCLT,pXSDERW,DK14,ABO14,kosERWPC,
  kpERWMCS,PT17,Tra18,kmpCRW,KMP23,MM24} and references
therein. Reviewing methods used to prove scaling limits for these
models, one naturally arrives at the question of
whether a GRKT uniquely identifies the scaling limit and might be, in
itself, a sufficient tool for establishing a functional limit
theorem. A negative answer to this question was given in \cite[Theorem
1.2]{KMP23}. The current paper suggests that for several classes of
models on $\Z$ a {\em joint} GRKT might be just the right tool in the
sense that the joint GRKT can be ``inverted'' to imply the convergence
of the rescaled random walks to a limit process that it identifies
uniquely.

The paper is organized as follows. In Section~\ref{Prelim} we
summarize properties of the TSRM and its local times as well as of
their discrete counterparts. After introducing the necessary notation
we state the joint GRKT, Theorem~\ref{thm:JRK-TSAW} and its
applications, in particular tightness,
Proposition~\ref{prop:XHtight}. In Section ~\ref{Outline}, using the
results stated in the previous section we prove
Theorem~\ref{thm:TSAW-FLT}. A proof of Theorem~\ref{thm:JRK-TSAW} is
given in Section~\ref{sec:JRK}. Section~\ref{sec:tight} contains a
proof of tightness results. Four appendices provide the necessary
auxiliary constructions and proofs of some technical results used in
the main body of the paper.

\section{Preliminaries, the joint GRKT and its consequences}\label{Prelim}

\subsection{Properties of TSRM and its local time}
The $\R\times\R_+$-valued random space-filling curve
$t\mapsto (\mathfrak{X}(t),\mathfrak{H}(t)),\ t\ge 0$, was constructed
in \cite[Section 3]{TW98}. The TSRM $\mathfrak{X}(\cdot)$ is defined
as the projection on the first coordinate, \cite[(3.15)]{TW98}. A very
rough outline of the construction is as follows. The authors first
build a system of coalescing reflected/absorbed independent Brownian
motions running forward and backward from every point
$(x,h)\in\R\times(0,\infty)$, see Figure~\ref{fig:RAB}.
\begin{figure}[h]
 \includegraphics[width=0.6\textwidth]{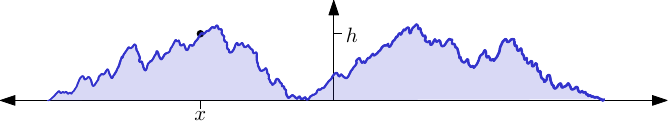} 
 \caption{A reflected/absorbed two-sided Brownian motion started from $(x,h)$ with $x<0$ . It is reflected on the interval $[x,0]$ and absorbed on $\R\setminus[x,0]$. $T(x,h)$ is the area between the curve and the horizontal axis. }\label{fig:RAB}
\end{figure}
They show that, almost surely, the areas $T(x,h)$ enclosed by the
forward and backward reflected/absorbed Brownian motions
starting from $(x,h)$ and the space axis are
strictly ordered and dense in $\R_+$, i.e., almost surely, the random
map $\R\times (0,\infty)\ni (x,h)\mapsto T(x,h)\in \R_+$ is
injective and its image is dense in $\R_+$. The process
$(\mathfrak{X}(t),\mathfrak{H}(t)), \ t\ge 0$, is then defined by
inverting this map and extending its inverse to the whole $\R_+$ by
continuity. The area $T(x,h)$ is exactly the time by which the TSRM
accumulates the local time $h$ at location $x$ and the forward and
backward curves starting from the point $(x,h)$
give the local time profile of the process at time $T(x,h)$ so that
$\mathfrak{L}(T(x,h),x)=h=\mathfrak{H}(T(x,h))$.

The system of independent coalescing Brownian motions starting from
every point in $\R^2$ is known as a ``Brownian web'' (see
\cite{Arr79,FINR02,FINR04,SSS17} and references therein).  Since we
will need neither the details of the above construction nor the
Brownian web in the present work, we will restrict ourselves to a
brief overview of some of the important properties of the TSRM and its
local time that will be used later in the paper. We refer to
\cite{TW98} for proofs and additional details.

\noindent $\circ$ {\em Continuity and scaling.} The process $t\mapsto \mathfrak{X}(t)$, $t\ge 0$, is almost surely continuous and satisfies the following scaling relation: 
 \begin{equation}\label{TSRM-scaling}
  (\mathfrak{X}(ct) )_{t\geq 0} \overset{\text{Law}}{=} ( c^{2/3} \mathfrak{X}(t) )_{t\geq 0}, \qquad \forall c\ge 0.  
\end{equation}

\noindent $\circ$ {\em Local times and occupation time formulas.}
There is a local time process
$(\mathfrak{L}(t,x))_{t\geq 0, x \in \R}$ for $\mathfrak{X}$ in the
sense that, almost surely, for all $t\geq 0$ the occupation time
measure $\mu_t$ of $\mathfrak{X}$ on the time interval $[0,t]$ defined
by
 \[
  \mu_t(A) = \int_0^t \ind{\mathfrak{X}(s) \in A} \, ds, \qquad \text{for all Borel } A \subset \R, 
 \]
 is absolutely continuous with respect to the Lebesgue measure on the
 Borel $\s$-field with density
 $\mathfrak{L}(t,\cdot)$. The mapping $t\mapsto \mathfrak{L}(t,\cdot)$
 is non-decreasing continuous from $[0,\infty)$ to the space of
 continuous functions with compact support with topology induced by
 uniform convergence on compact intervals.

 The process $\mathfrak{H}(t)=\mathfrak{L}(t,\mathfrak{X}(t))$,
 $t\ge 0$, is almost surely continuous and we have the following generalization of the occupation time formula: almost surely, for
 any bounded, measurable function $\phi:\R\times\R_+\to \R$ and any
 $t\ge 0$, 
 \begin{equation*}
   \int_0^t \phi(\mathfrak{X}(s),\mathfrak{H}(s))\, ds = \int_\R \int_0^t\phi(x,\mathfrak{L}(s,x)) d_s\mathfrak{L}(s,x) \, dx=\int_\R\int_{[0,\mathfrak{L}(t,x))}\phi(x,h)\,dh\,dx. 
 \end{equation*}
 The first equality is an application of \cite[Theorem 4.2(ii)]{TW98}
 with $g(t,x)=\phi(x,\mathfrak{L}(t,x))$ and the second is the change
 of variables.  In fact, we will need a
 multi-dimensional version of this formula which can be obtained from
 the above by using Fubini's theorem: almost surely, for any bounded,
 measurable $f:\R^k\times\R^k_+ \to \R$ and any
 $t_1,t_2,\ldots,t_k\ge 0$,
 \begin{multline}
   \label{otf}
   \int_{\prod_{i=1}^k [0,t_i]} \phi(\mathfrak{X}(s_1),\dots,\mathfrak{X}(s_k),\mathfrak{H}(s_1),\dots,\mathfrak{H}(s_k)) \, ds_1 \ldots ds_k \\ = \int_{\R^k}\int_{\prod_{i=1}^k [0,\mathfrak{L}(t_i,x_i))} \phi(x_1,\ldots,x_k,h_1,\ldots,h_k) \,dh_1\ldots dh_k\, dx_1 \ldots dx_k.  
 \end{multline}
 
\noindent $\circ$ {\em Ray-Knight curves and their merge and absorption times.} For any $(x,h) \in \R\times\R_+$ let $\mathfrak{t}_{x,h} $ be the first time when the local time of $\mathfrak{X}$ at $x$ exceeds $h$, 
\begin{equation}
  \label{txh}
  \mathfrak{t}_{x,h} = \inf\{ t\geq 0: \, \mathfrak{L}(t,x) > h \},
\end{equation}
and define the local time profile of $\mathfrak{X}$ at time $\mathfrak{t}_{x,h}$ by
\begin{equation}\label{RK-TSRM}
 \Lambda_{x,h}(y) = \mathfrak{L}\left( \mathfrak{t}_{x,h}, y \right), \quad y \in \R. 
\end{equation}
We note that $\Lambda_{x,h}(x)=h$ and will refer to
$\Lambda_{x,h}(\cdot)$ as a Ray-Knight curve indexed by $(x,h)$.

For $x,x'\in \R$ and $h,h' \geq 0$ let
 \[
  \mathfrak{m}_{(x,h),(x',h')}^+ 
  =  \inf\left\{ y \geq x\vee x': \Lambda_{x,h}(y) = \Lambda_{x',h'}(y) \right\},
 \]
and 
\[
 \mathfrak{m}_{(x,h),(x',h')}^- 
  =  \sup\left\{ y \leq x\wedge x': \Lambda_{x,h}(y) = \Lambda_{x',h'}(y) \right\}. 
\]
Thus, $\mathfrak{m}_{(x,h),(x',h')}^+$ and $\mathfrak{m}_{(x,h),(x',h')}^-$ are the locations right of $x\vee x'$ and left of $x\wedge x'$, respectively, where the curves $\Lambda_{x,h}$ and $\Lambda_{x',h'}$ coalesce. 
They also represent the maximal and minimal points, respectively, that the process $\mathfrak{X}$ reaches between the times $\mathfrak{t}_{x,h}$ and $\mathfrak{t}_{x',h'}$.
Since $\Lambda_{0,0}(y) \equiv 0$, the values $\mathfrak{m}_{(0,0),(x,h)}^+=:\mathfrak{m}_{x,h}^+$ and $\mathfrak{m}_{(0,0),(x,h)}^-=:\mathfrak{m}_{x,h}^-$ give the locations to the right of $x\vee 0$ and left of $x\wedge 0$, respectively, where the curve $\Lambda_{x,h}$ is absorbed at zero, and thus
are also the maximum and the minimum of the TSRM up to $\mathfrak{t}_{x,h}$ so that we can write  
\begin{equation}\label{txh-integral}
 \mathfrak{t}_{x,h} = \int_{\R}\Lambda_{x,h}(y) \, dy=\int_{\mathfrak{m}_{x,h}^-}^{\mathfrak{m}_{x,h}^+}\Lambda_{x,h}(y) \, dy, \qquad \forall (x,h) \in \R\times\R_+. 
\end{equation}

\noindent $\circ$ {\em Joint Ray-Knight Theorems for the TSRM.}  For
any finite set of points $(x_1,h_1), (x_2,h_2),\ldots,$
$(x_k,h_k) \in \R \times [0,\infty)$, the curves
$(\Lambda_{x_i,h_i}(\cdot))_{i\leq k}$ are a family of
coalescing reflected/ab\-sorbed Brownian
motions. We will need the following more precise description.
\begin{enumerate}[label={\tiny $\bullet$},leftmargin=*]  
\item {\em Marginal curve distribution.} For each $i$ the process
  $\Lambda_{x_i,h_i}(\cdot)$ is a two-sided reflected/ab\-sorbed
  Brownian motion. To the right of $x_i$ the process is a
  one-dimensional Brownian motion started with initial value
  $\Lambda_{x_i,h_i}(x_i) = h_i$ which is reflected at 0 on the
  interval $[x_i,0]$ if $x_i < 0$ and absorbed at 0 on
  $[0\vee x_i, \infty)$. To the left of $x_i$ the process is a
  one-dimensional Brownian motion (which is independent of the process
  to the right of $x_i$) also started at $h_i$, reflected at 0 on the
  interval $[0,x_i]$ if $x_i > 0$ and absorbed at 0 on the interval
  $(-\infty,x_i \wedge 0]$. See 
  Figure~\ref{fig:RAB}.
\item {\em Joint forward curve distribution.}  For each $i$ let
  $\Lambda_{x_i,h_i}^+(\cdot)$ denote the part of the curve
  $\Lambda_{x_i,h_i}(\cdot)$ restricted to the interval
  $[x_i,\infty)$. Then the joint distribution of the curves
  $(\Lambda_{x_i,h_i}^+(\cdot))_{i\leq k}$ is as follows: the curves
  evolve as reflected/absorbed Brownian motions (as described above)
  which are independent until they meet and which
  coalesce upon their first intersection.  See
  Figure~\ref{fig:CRAB}. 
\begin{figure}[h]
 \includegraphics[width=0.6\textwidth]{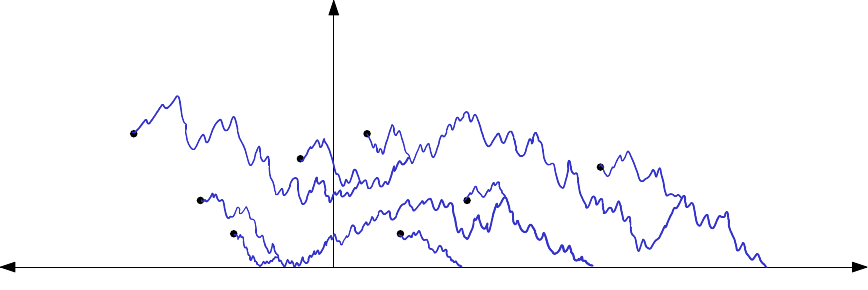}
 \caption{A system of independent coalescing, reflected/absorbed
   Brownian motions
   $(\Lambda_{x_i,h_i}^+(\cdot))_{i\leq k}$.}\label{fig:CRAB}
\end{figure}
  \item {\em Backward curves and duality.}  For
    $(x,h)\in\R\times(0,\infty)$ define the backward curve
    by
    \[\Lambda^-_{x,h}(y)=\sup\{h'>0: \ \Lambda_{y,h'}(x)<h\},\quad
      y\le x,\] where $\sup\emptyset:=0$, and
    $\Lambda^-_{x,0}(y)=\inf\{\Lambda^-_{x,h}(y): h>0\}$, $y\le
    x$. 
    
    Let $\mathbb{F}^+:=\{(x,h,y)\in\R^3:y\ge x, h>0\}$.
    The duality theorem (\cite[Theorem 2.3]{TW98}) states that the
    two processes
    \[\mathbb{F}^+\ni (x,h,y)\mapsto
      \Lambda^+_{x,h}(y)\in\R_+\quad \text{and}\quad \mathbb{F}^+\ni
      (x,h,y)\mapsto \Lambda^-_{-x,h}(-y)\in\R_+\] are identical in
    law. We note two important consequences. First, 
    the backward curves $\Lambda^-_{x_i,h_i}(\cdot)$
    are Brownian motions reflected at $0$ on $[0,x_i]$ if $x_i>0$ and
    absorbed at $0$ on $(-\infty, x_i\wedge 0]$ which are independent
    until they meet and which coalesce upon their first
    intersection. Second, while the families
    $(\Lambda^+_{x_i,h_i}(\cdot))_{1\le i\le k}$ and
    $(\Lambda^-_{x_i,h_i}(\cdot))_{1\le i\le k}$ are not independent
    in general, the
    joint distribution of one of these families uniquely determines
    the joint distribution of the other and, thus, the joint
    distribution of 
    $(\Lambda_{x_i,h_i}(\cdot))_{1\le i\le k}$. We refer to
    \cite{STW00} for further details about the relationship between
    the forward and backward curves.  
\end{enumerate}

\begin{remark}
   The inverse local times $\mathfrak{t}_{x,h}$ (see \eqref{txh}) that
   we use in our paper correspond to $T^+(x,h)$ in \cite{TW98}.
   Many of the results in \cite{TW98} are stated instead in terms of
   the stopping times
   $T(x,h) = \inf\{t\geq 0: \, \mathfrak{L}(t,x) = h \}$ for
   $x \in \R$ and $h>0$.  The only difference between $T^+(x,h)$ and
   $T(x,h)$ is that $h\mapsto T(x,h)$ is left continuous while
   $h \mapsto T^+(x,h)$ is right continuous. In particular, for a
   fixed $(x,h)$ with $h>0$ we have that $T(x,h) = T^+(x,h)$, almost
   surely.  The choice of $\mathfrak{t}_{x,h} = T^+(x,h)$ is more
   convenient for our purposes, in part because then the Ray-Knight
   Theorems in \cite[Theorem 4.3]{TW98} are true for all
   $h\geq 0$ rather than just for $h>0$.
 \end{remark}

\subsection{Discrete Ray-Knight curves and the joint GRKT}
Recall the definition of local times \eqref{Lt} and let
\[\tau_{k,m} = \inf\{ n\geq 0: L(n,k) > m \},\quad k \in\Z, \ m\in\Z_+, \] be
the time of the $(m+1)$-th visit of the TSAW to $k$.  The joint GRKT
for the TSAW will describe the convergence of the rescaled local time
profile of the random walk at finitely many stopping times
$\tau_{k,m}$ with $k$ of order $n$ and $m$ of order $\sqrt{n}$. To
this end, for any $(x,h)\in \R\times\R_+$ we define the rescaled local
times curve $\Lambda_{x,h}^n(\cdot)$ at time
$\tau_{\fl{xn},\fl{2\s h\sqrt{n}}}$ by
\begin{equation}\label{dRKpath}
 \Lambda_{x,h}^n(y) = \frac{L\left(\tau_{\fl{xn},\fl{2\s h\sqrt{n}}}, \fl{yn} \right)}{2 \sigma\sqrt{n}}, 
\end{equation}
where $\sigma$ is given by \eqref{sigma}.  The joint GRKT stated below
gives the joint weak convergence of a finite set of curves
$\Lambda_{x_i,h_i}^n(\cdot)$, $1\le i\le N$, together with the
locations where these curves coalesce or get absorbed.  Hence, we
shall need notation for recording this information. For any
$k,k' \in \Z$ and $m,m' \geq 0$ let
\begin{align*}
  \mu_{(k,m),(k',m')}^+ &= \max\left\{ X_i: \tau_{k,m} \wedge \tau_{k',m'}\leq i \leq \tau_{k,m} \vee \tau_{k',m'} \right\}\\ &=\min\left\{ j \geq k \vee k': \mathcal{E}_{\tau_{k,m}}^+(j) = \mathcal{E}_{\tau_{k',m'}}^+(j) \right\}  
\shortintertext{and} 
  \mu_{(k,m),(k',m')}^- & = \min\left\{ X_i: \tau_{k,m} \wedge \tau_{k',m'} \leq i \leq \tau_{k,m} \vee \tau_{k',m'} \right\}\\ &= \max\left\{ j \leq k \wedge k': \mathcal{E}_{\tau_{k,m}}^-(j) = \mathcal{E}_{\tau_{k',m'}}^-(j) \right\}
\end{align*}
 be the farthest the random walk ever is to the right or left between the stopping times $\tau_{k,m}$ and $\tau_{k',m'}$. 
 Note that these random variables differ by at most 1 from the first location to the right of $k \vee k'$ or left of $k\wedge k'$, respectively, where the local time processes at times $\tau_{k,m}$ and $\tau_{k',m'}$ are equal. 
 Since $\Lambda^n_{0,0}(y) \equiv 0$, the values
 $\mu^+_{(0,0),(k,m)}=:\mu^+_{k,m}$ and
 $\mu^-_{(0,0),(k,m)}=:\mu^-_{k,m}$ give the locations to the right
 of $k\vee 0$ and left of $k\wedge 0$, respectively, where the curve
 $\Lambda^n_{k,m}$ is absorbed at zero. These values are also the
 maximum and the minimum of the TSAW up to time $\tau_{k,m}$.

 To lighten the notation, for $x,x'\in\R$ and $h,h'\in\R_+$ we let
 \begin{align}\label{tnxh}
   \tau^n_{x,h}  &= \tau_{\fl{xn},\fl{2\s h\sqrt{n}}}; \\ \label{mnxh} \mu^{n,\pm}_{x,h} & = n^{-1}\mu^\pm_{(\fl{xn},\fl{2\s h \sqrt{n}})},\quad
   \mu^{n,\pm}_{(x,h),(x',h')} = n^{-1}\mu^\pm_{(\fl{xn},\fl{2\s h \sqrt{n}}),(\fl{x'n},\fl{2\s h' \sqrt{n}})}. 
 \end{align}

\begin{theorem}[Joint GRKT]\label{thm:JRK-TSAW}
  For any $N\in\N$ and any choice of $(x_i,h_i)\in \R \times \R_+$,
  $1\le i\le N$, the joint distribution of the processes
  $\Lambda_{x_i,h_i}^n(\cdot)$, the rescaled endpoints
  $\mu^{n,*}_{x_i, h_i} $, and merging points
  $\mu^{n,*}_{(x_i,h_i),(x_j,h_j)}$,
  $1\le i\leq N, \, i<j\le N,\, * \in \{+,-\}$, converges as
  $n\to\infty$ to the joint distribution of
  $\Lambda_{x_i,h_i}(\cdot)$, $\mathfrak{m}^*_{x_i,h_i}$, and
  $\mathfrak{m}^*_{(x_i,h_i),(x_j,h_j)} $,
  $1\le i\leq N, \, i<j\le N,\, * \in \mathcal\{+,-\}$. The
    convergence is in the product space
    ${\cal D}^N \times \R^{2N} \times \R^{N(N-1)}$, where ${\cal D}$ is equipped with the
    topology of uniform convergence on compact subsets.
\end{theorem}

\begin{remark} Starting from this result, it seems natural to
    prove Theorem~\ref{thm:TSAW-FLT} using an approach similar to that
    in \cite{NR06}, which relies on the well-developed Brownian web
    framework.  As a first step, the convergence in
    Theorem~\ref{thm:JRK-TSAW} would be extended to the convergence of
    the entire family of rescaled discrete local time curves
    $\{\Lambda_{x,h}^n(\cdot)\}_{(x,h) \in \frac{1}{n}\Z \times
      \frac{1}{2\s\sqrt{n}}\Z_+}$, the so-called discrete web, to the
    Brownian web (see \cite{FINR04} or \cite{SSS17} for precise
    definitions of the appropriate space and topology for this
    convergence). Next, one would use properties of the Brownian web,
    its dual, and the double Brownian web to establish tightness and
    convergence of the process in \eqref{discreteTSRM}
    to $(\mathfrak{X}(\cdot),\mathfrak{H}(\cdot))$.  Since our primary
    interest is to show the convergence of the rescaled TSAW to the
    TSRM, we regard bypassing the convergence of discrete webs to the
    Brownian web -- and instead proving the convergence to the TSRM
    directly, using only Theorem \ref{thm:JRK-TSAW} and properties of
    the TSRM -- as an advantage of our method. Moreover, our
    tightness argument relies solely on Theorem~\ref{thm:JRK-TSAW} and
    is readily adaptable to some walks for which the scaling limit has
    not yet been constructed, such as polynomially self-repelling or
    generalized TSAW models.
\end{remark}

For $N=1$ Theorem \ref{thm:JRK-TSAW} was proven in \cite[Theorem 1]{tTSAW}. We state
it and discuss a minor difference with our formulation in Section
\ref{sec:JRK}. In \cite{TW98} the authors suggested that the methods
of \cite{tTSAW} can be used to show the joint convergence of the
rescaled edge local times to independent coalescing reflected/absorbed
Brownian motions, see \cite[(1.13)]{TW98}. We show that this is indeed
the case. Yet we would like to point out that the generalization from
$N=1$ to an arbitrary $N$ is non-trivial and quite technical due to the
included convergence of the merge and absorption points. The reader is
referred to Section \ref{sec:JRK} for a proof.

\subsection{Consequences of the joint GRKT}

We conclude this section with two consequences of
Theorem~\ref{thm:JRK-TSAW} that will be the key tools used in our
proof of Theorem \ref{thm:TSAW-FLT}.

The first one is the weak convergence of the rescaled inverse local
times of the TSAW to those of the true self-repelling motion.

\begin{corollary}\label{cor:tau-lim}
  For any $N\geq 1$ and any choice of $(x_i,h_i)\in \R \times \R_+$,
  $1\le i\le N$, we have
\[
 \left( \frac{\tau^n_{x_i,h_i}}{2 \s n^{3/2}} \right)_{i\leq N} 
 \underset{n\to\infty}{\Longrightarrow}
 \left( \mathfrak{t}_{x_i,h_i} \right)_{i\leq N}. 
\]
\end{corollary}
\begin{proof}
Since $\tau_{k,m} + 1 = \sum_{j \in \Z} L(\tau_{k,m},i)$, we see that 
\begin{align}
 \frac{\tau^n_{x_i,h_i}}{2\s n^{3/2}} + \frac{1}{2\s n^{3/2}}
 &= \frac{1}{2\s n^{3/2}}\sum_{k\in \Z} L\left(\tau^n_{x_i,h_i}, k \right) 
 \nonumber \\
  & \overset{\eqref{dRKpath}}{=} \frac{1}{n} \sum_{k \in \Z} \Lambda^n_{x_i,h_i}\left(k/n\right)  = \int_\R \Lambda^n_{x_i,h_i}(y) \, dy =\int_{\mu^{n,-}_{x_i,h_i}}^{\mu^{n,+}_{x_i,h_i}}\Lambda^n_{x_i,h_i}(y) \, dy. \label{txn-integral}
\end{align}
The statement now follows from \eqref{txh-integral}
  and Theorem \ref{thm:JRK-TSAW}, from which we need the joint
  convergence of $(\Lambda^n_{x_i,h_i}(\cdot))_{1\le i\le N}$ and the
  tightness of distributions of $\mu_{x_i,h_i}^{n,\pm}$,
  $1\le i\le N$. We omit the details as they are standard.
\end{proof}

An immediate consequence of Corollary \ref{cor:tau-lim} is convergence of the multi-dimensional Laplace transforms of these rescaled hitting times. That is, for any fixed $k\in\N$, $(x_i,h_i)\in \R\times\R_+$, and $\lambda_i > 0$, $i\in\{1,2,\ldots,k\}$, we have 
\[
 \lim_{n\to\infty} E\left[ e^{-n^{-3/2}\sum_{i=1}^k \lambda_i  \tau^n_{x_i ,h_i} } \right] = E\left[ e^{-\sum_{i=1}^k 2\s \l_i \mathfrak{t}_{x_i,h_i}}  \right].  
\]
The above suggests that scaling of the time by $n^{3/2}$ rather than by $n$ might be more natural for the methods we are going to use. To this end, we define for $t\ge 0$
\begin{equation}
  \label{XH}
  \mathcal{X}_n(t) = \frac{X_{\fl{t n^{3/2}}}}{(2\s)^{-2/3} n},\ \ \mathcal{H}_n(t)=\frac{L(\fl{t n^{3/2}}, X_{\fl{t n^{3/2}}})}{(2\sigma)^{2/3} \sqrt{n}},
\end{equation}
where $\sigma$ is given by \eqref{sigma}. At the beginning of the next section we shall briefly argue that the
weak convergence of ${\cal D}^2_+$-valued random variables
$(\mathcal{X}_n(\cdot),\mathcal{H}_n(\cdot))$, $n\in\N$, is equivalent
to the convergence claimed in Theorem~\ref{thm:TSAW-FLT}.
 
The second main consequence of Theorem~\ref{thm:JRK-TSAW} is tightness of the rescaled paths of the TSAW and its local times.

\begin{proposition}\label{prop:XHtight}
 The sequence of ${\cal D}^2_+$-valued random variables $(\mathcal{X}_n,\mathcal{H}_n)$, $n\in\N$, is tight (with respect to the Skorokhod $J_1$ topology on ${\cal D}^2_+$). Moreover, any subsequential limit is concentrated on paths that are continuous.  
\end{proposition}

Proposition \ref{prop:XHtight} will be proved in Section
\ref{sec:tight} using Theorem~\ref{thm:JRK-TSAW}.  Below we will simply
give some intuition into the role that the joint GRKT plays in proving
tightness.

We first note how the joint GRKT for the TSRM $\mathfrak{X}$ can be
used to control the continuity of $\mathfrak{X}$.  If there exist
times $s,t$ with $|s-t|<\d$ and
$|\mathfrak{X}(t)-\mathfrak{X}(s)| > \e$, then there must be inverse
local times $\mathfrak{t}_{x,h}$ and $\mathfrak{t}_{x',h'}$ with
$|x-x'|>\e$ and
\[
 \d > |\mathfrak{t}_{x,h} - \mathfrak{t}_{x',h'} | 
 = \int_\R \left| \Lambda_{x,h}(y) - \Lambda_{x',h'}(y) \right| \, dy
\]
However, since $\Lambda_{x,h}(\cdot)$ and $\Lambda_{x',h'}$ are
independent coalescing reflected/absorbed Brownian motions, it is very
unlikely that the curves $\Lambda_{x,h}$ and $\Lambda_{x',h'}$ can
have a small area between them if $x$ and $x'$ are far apart (it
should be noted that the curves can coalesce only outside of the
interval $[x\wedge x', x\vee x']$).  The same ideas can be applied to
the random walk and the corresponding curves $\Lambda^n_{x,h}$ which
appear in Theorem \ref{thm:JRK-TSAW}. Combining the tightness of the
sequence of the spatial coordinates with properties of the local times
we then show the tightness of $({\cal H}_n)_{n\in\N}$ and complete
the proof of Proposition~\ref{prop:XHtight}.

\section{Proof of Theorem~\ref{thm:TSAW-FLT}}\label{Outline}

In this section, we prove Theorem \ref{thm:TSAW-FLT} using the results
stated in the previous section. As we remarked above, it suffices to
prove the following statement.

\begin{theorem}\label{thm:TSAW-FLL}
  The sequence of ${\cal D}^2_+$-valued random variables
  $(\mathcal{X}_n(\cdot),\mathcal{H}_n(\cdot))$, $n\in\N$, defined in
  \eqref{XH} 
  converges weakly as $n\to\infty$ to 
  $(\mathfrak{X}(\cdot),\mathfrak{H}(\cdot))$.
\end{theorem}

Indeed, let $k_n=\fl{n^{2/3}}$ and $t_n=tn/k^{3/2}_n$. Writing
\begin{equation*}
  \left(\frac{X_{\fl{tn}}}{(2\sigma)^{-2/3}n^{2/3}},\frac{L\left(\fl{tn}, X_{\fl{tn}}\right)}{(2\sigma)^{2/3} n^{1/3}}\right)=
  \left(\frac{k_n}{n^{2/3}}{\cal X}_{k_n}(t_n),\frac{k_n^{1/2}}{n^{1/3}}{\cal H}_{k_n}(t_n)\right)
\end{equation*}
and using the fact that $k_n/n^{2/3}\to 1$ and $0\le t_n-t\to 0$ as
$n\to \infty$ locally uniformly in $t\ge 0$, we see that for every
$T>0$ the Skorokhod distance between the last process and
$({\cal X}_{k_n}(t),{\cal H}_{k_n}(t))$ 
on the interval $[0,T]$
goes to $0$ as
$n\to\infty$. We conclude that Theorem~\ref{thm:TSAW-FLL} implies
Theorem~\ref{thm:TSAW-FLT}.

\subsection{Proof of Theorem \ref{thm:TSAW-FLL}}
The first step in the proof of Theorem \ref{thm:TSAW-FLL} is to show
the convergence of finite dimensional distributions for the processes
$(\mathcal{X}_n(\cdot),{\cal H}_n(\cdot))_{n\in\N}$ sampled not at
deterministic times but at independent exponential times.  Without
loss of generality we expand the probability spaces of the random walk
and the process $\mathfrak{X}$ to include the exponential random
variables that are also independent of both the TSAW and the process
$\mathfrak{X}$.

\begin{theorem}[Joint weak limit at independent geometric stopping times]\label{th:JWLGST}
  For any choice of parameters $\l_1,\l_2,\ldots,\l_k > 0$, let
  $\gamma_1, \gamma_2,\ldots,\gamma_k$ be independent random variables
  with $\gamma_i \sim \text{Exp}(\l_i)$ that are also independent of
  both the TSAW and the true self-repelling motion $\mathfrak{X}$.
  Then,
\[\left( \mathcal{X}_n(\gamma_i),  \mathcal{H}_n(\gamma_i) \right)_{1\leq i \leq k}\ 
 \underset{n\to\infty}{\Longrightarrow} 
\left( \mathfrak{X}(\gamma_i), \mathfrak{H}(\gamma_i) \right)_{1\leq i \leq k}.
 %
\]
\end{theorem}

\begin{remark}
  The proof of Theorem \ref{th:JWLGST} is essentially a multi-point extension of the argument in the proof of Theorem 3 in \cite{tTSAW} where we use Corollary~\ref{cor:tau-lim} instead of \cite[Corollary 2]{tTSAW}. 
  Thanks to \cite{TW98} and
  \eqref{otf}, we can identify the density of the limiting random
  vector as that of $(\mathfrak{X}, \mathfrak{H})$ stopped at random
  exponential times (recall that \cite{tTSAW} preceded the
  construction of the TSRM done in \cite{TW98}). The convergence
  result for $k=1$ is contained already in \cite{tTSAW}, even though
  it is not stated explicitly. The author establishes a local limit
  theorem for the pair
  $(n^{-2/3}X_{\fl{\gamma_1 n}}, n^{-1/3}H_{\fl{\gamma_1 n}})_{n\ge
    1}$ (see \cite[(6.6)]{tTSAW}) from which the convergence for $k=1$
  readily follows. 
  A local limit theorem for
  $(n^{-2/3}X_{\fl{\gamma_1 n}})_{n\ge 1}$ can also be obtained by
  providing appropriate uniform tail estimates and an application of
  the dominated convergence theorem. Since we do not need this result
  we omit the details. We remark that the proof of the local limit theorem
  for $(n^{-2/3}X_{\fl{\gamma_1 n}})_{n\ge 1}$ given in \cite{tTSAW}
  which uses only Fatou's lemma (\cite[p.\,1555]{tTSAW}) appears to be
  incomplete as only a liminf rather than the claimed limit is proved.
\end{remark}

\begin{proof}
Let $H_n=L(n,X_n)$ and recall \eqref{XH}. We shall first prove that
\begin{equation}\label{XgHg-local}
\begin{split}
 \lim_{n\to\infty} &n^{3k/2} P\left( X_{\fl{\gamma_i n^{3/2}}} = \fl{(2\s)^{-2/3} a_i n}, \, H_{\fl{\gamma_i n^{3/2}}} = \fl{(2\s)^{2/3} h_i \sqrt{n}} + 1, \, \forall i\leq k \right) \\
 &= E\left[ \prod_{i=1}^k \l_i e^{-\l_i \mathfrak{t}_{a_i,h_i}} \right],
 \qquad \forall a_1,a_2,\ldots,a_k \in \R, \, \forall h_1,h_2,\ldots,h_k\geq 0, 
\end{split}
\end{equation}
and then show that the right-hand side is the joint density of
$\left( \mathfrak{X}(\gamma_i), \mathfrak{H}(\gamma_i)
\right)$, $1\leq i \leq k$. Indeed, we can interpret the expression
under the limit as a density of a $2k$-dimensional random vector with
a piecewise constant density, so that the above convergence of
densities by Scheffe's lemma gives us the weak convergence of this
sequence of $2k$-dimensional random vectors to
$\left( \mathfrak{X}(\gamma_i), \mathfrak{H}(\gamma_i) 
\right)$, $1\leq i \leq k$. The statement of the theorem follows by
noticing that the random variables $\left( \mathcal{X}_n(\gamma_i), \mathcal{H}_n(\gamma_i) \right)$, $i\leq k$ can be coupled with joint random variables whose density is given by the left side of \eqref{XgHg-local} in such a way that the maximum difference of any coordinate is $\mathcal{O}(n^{-1/2})$.

To prove \eqref{XgHg-local}, note that $\fl{\gamma_i n^{3/2}} \sim \text{Geo}(1-e^{-\l_i n^{-3/2}})$ so that by conditioning on the values of these independent geometric random variables we get 
\begin{align*}
&P\left( X_{\fl{\gamma_i n^{3/2}}} = \flo{\tfrac{a_i n}{(2\sigma)^{2/3}}}, H_{\fl{\gamma_i n^{3/2}}} = \fl{(2\s)^{2/3} h_i \sqrt{n}} + 1, \, \forall i\leq k \right) \\
&= \sum_{j_1,j_2,\ldots,j_k \geq 0}
\left( \prod_{i=1}^k (1-e^{-\l_i n^{-3/2}})\,e^{-\l_i n^{-3/2}j_i } \right)\\ &\makebox[4.7cm]{\ } \times P\left( X_{j_i} = \flo{\tfrac{a_i n}{(2\sigma)^{2/3}}}, H_{j_i} = \fl{(2\s)^{2/3} h_i \sqrt{n} }+1, \, \forall i\leq k \right)\\
&= \sum_{j_1,\ldots,j_k \geq 0}
\left( \prod_{i=1}^k (1-e^{-\l_i n^{-3/2}})\,e^{-\l_i n^{-3/2}j_i} \right) P\left( \tau_{\fl{(2\s)^{-2/3} a_i n}, \fl{(2\s)^{2/3} h_i \sqrt{n}}} = j_i , \, \forall i\leq k \right) \\
&= \prod_{i=1}^k (1-e^{-\l_i n^{-3/2}}) E\left[ \prod_{i=1}^k \exp\left\{ -\l_i n^{-3/2}\tau_{\fl{(2\s)^{-2/3} a_i n}, \fl{(2\s)^{2/3} h_i \sqrt{n}}}  \right\} \right]. 
\end{align*}
Therefore, it follows from Corollary \ref{cor:tau-lim} that 
\begin{align*}
 &\lim_{n\to\infty} n^{3k/2} P\left( X_{\fl{\gamma_i n^{3/2}}} = \fl{(2\s)^{-2/3} a_i n}, \, H_{\fl{\gamma_i n^{3/2}}} = \fl{(2\s)^{2/3} h_i \sqrt{n}}+1, \, \forall i\leq k \right) \\
 &= \lim_{n\to\infty} \left( \prod_{i=1}^k n^{3/2}(1-e^{-\l_i n^{-3/2}}) \right)  E\left[ \prod_{i=1}^k \exp\left\{ -\l_i n^{-3/2} \tau_{\fl{(2\s)^{-2/3} a_i n}, \fl{(2\s)^{2/3} h_i \sqrt{n}}}  \right\} \right] \\
 &= \left( \prod_{i=1}^k \l_i \right) E\left[ \prod_{i=1}^k \exp\left\{  -2\s \l_i \mathfrak{t}_{(2\s)^{-2/3}a_i, (2\s)^{-1/3}h_i} \right\} \right]= \left( \prod_{i=1}^k \l_i \right) E\left[ \prod_{i=1}^k e^{ - \l_i \mathfrak{t}_{a_i, h_i}} \right], 
\end{align*}
where in the last equality we used that
$( \mathfrak{t}_{c a_i, \sqrt{c} h_i} )_{i\leq k}
\overset{\text{Law}}{=} ( c^{3/2}\, \mathfrak{t}_{a_i,h_i} )_{i\leq
  k}$ for any $c>0$, which follows from \eqref{txh-integral}, the fact
that the curves $(\Lambda_{a_i,h_i}(\cdot))_{i\leq k}$ are independent
coalescing reflected/absorbed Brownian motions, and then the scaling
property of Brownian motion.

We have completed the proof of \eqref{XgHg-local}, and now it remains to show that \linebreak
$((a_1,h_1),\ldots,(a_k,h_k)) \mapsto E\left[ \prod_{i=1}^k \l_i e^{-\l_i \mathfrak{t}_{a_i,h_i}} \right]$ is the joint density for 
$\left( \mathfrak{X}(\gamma_i), \mathfrak{H}(\gamma_i)
\right)$, $1\leq i \leq k$.
To this end, fix a bounded continuous function $\phi: \R^k \times \R_+^k \to \R $ and note that (using the notation $\mathbf{a}=(a_1,\ldots,a_k)$,  $\mathbf{h} = (h_1,\ldots,h_k)$,  $\mathbf{da} = da_1\cdots da_k$ and $\mathbf{dh} = dh_1 \cdots dh_k$)
\begin{align*}
& \int_{\R^k} \int_{\R_+^k} \phi(\mathbf{a},\mathbf{h}) E\left[ \prod_{i=1}^k \l_i e^{-\l_i \mathfrak{t}_{a_i,h_i}} \right] \mathbf{dh} \, \mathbf{da}\\
&= \left( \prod_{i=1}^k \l_i \right) \int_{\R^k} \int_{\R_+^k} \phi(\mathbf{a},\mathbf{h}) P\left( \gamma_i > \mathfrak{t}_{a_i,h_i}, \, \forall i\leq k \right) \mathbf{dh} \, \mathbf{da} \\
&= \left( \prod_{i=1}^k \l_i \right) \int_{\R^k} \int_{\R_+^k} \phi(\mathbf{a},\mathbf{h})  P\left( \mathfrak{L}(\gamma_i,a_i) > h_i, \, \forall i\leq k \right) \mathbf{dh} \, \mathbf{da} \\
&= \left( \prod_{i=1}^k \l_i \right) \int_{\R^k} \int_{\R_+^k} \int_{\R_+^k} \left( \prod_{i=1}^k \l_i e^{-\l_i s_i} \right) \phi(\mathbf{a},\mathbf{h})  P\left( \mathfrak{L}(s_i,a_i) > h_i, \, \forall i\leq k \right) \mathbf{ds} \, \mathbf{dh} \, \mathbf{da} \\
&=  \int_{\R_+^k} \left( \prod_{i=1}^k \l_i^2 e^{-\l_i s_i} \right)  E\left[ \int_{\R^k} \int_{\R_+^k} \phi(\mathbf{a}, \mathbf{h}) \ind{h_i < \mathfrak{L}(s_i,a_i), \, \forall i \leq k} \,  \mathbf{dh} \, \mathbf{da} \right]   \mathbf{ds} \\
&= \int_{\R_+^k} \left( \prod_{i=1}^k \l_i^2 e^{-\l_i s_i} \right)  E\left[ \int_{\R^k} \int_{\prod_{i=1}^k [0,\mathfrak{L}(s_i,a_i))} \phi(\mathbf{a}, \mathbf{h}) \,  \mathbf{dh} \, \mathbf{da} \right]   \mathbf{ds}.
\end{align*}
By the generalization of the occupation time formula, see \eqref{otf}, 
\begin{equation*}\label{intLtHt}
 \int_{\R^k} \int_{\prod_{i=1}^k [0,\mathfrak{L}(s_i,a_i))}
 \phi(\mathbf{a}, \mathbf{h}) \, \mathbf{dh} \, \mathbf{da}
 = \int_{\prod_{i=1}^k [0,s_i] } \phi\left( \mathfrak{X}(t_1),\ldots, \mathfrak{X}(t_k), \mathfrak{H}(t_1),\ldots,\mathfrak{H}(t_k) \right) \, \mathbf{dt}.
\end{equation*}
Hence,
\begin{align*}
 & \int_{\R^k} \int_{\R_+^k} \phi(\mathbf{a},\mathbf{h}) E\left[ \prod_{i=1}^k \l_i e^{-\l_i \mathfrak{t}_{a_i,h_i}} \right] \mathbf{dh} \, \mathbf{da}\\
 &=\int_{\R_+^k} \left( \prod_{i=1}^k \l_i^2 e^{-\l_i s_i} \right) E\left[ \int_{ \prod_{i=1}^k [0,s_i] } \phi\left( \mathfrak{X}(t_1),\ldots, \mathfrak{X}(t_k), \mathfrak{H}(t_1),\ldots,\mathfrak{H}(t_k) \right) \, \mathbf{dt}  \right] \, \mathbf{ds} \\
 &= E\left[ \int_{\R_+^k}  \phi\left( \mathfrak{X}(t_1),\ldots, \mathfrak{X}(t_k), \mathfrak{H}(t_1),\ldots,\mathfrak{H}(t_k) \right) \left( \prod_{i=1}^k \int_{t_i}^\infty \l_i^2 e^{-\l_i s_i} \, ds_i \right) \mathbf{dt} \right] \\
 &= E\left[ \int_{\R_+^k}  \phi\left( \mathfrak{X}(t_1),\ldots, \mathfrak{X}(t_k), \mathfrak{H}(t_1),\ldots,\mathfrak{H}(t_k) \right) \left( \prod_{i=1}^k \l_i e^{-\l_i t_i} \right) \mathbf{dt} \right]\\
 &= E\left[ \phi\left( \mathfrak{X}(\gamma_1),\ldots, \mathfrak{X}(\gamma_k), \mathfrak{H}(\gamma_1),\ldots,\mathfrak{H}(\gamma_k) \right) \right].
\end{align*}
This completes the proof of the theorem.
\end{proof}
We are now ready to prove Theorem \ref{thm:TSAW-FLL}.
\begin{proof}[Proof of Theorem \ref{thm:TSAW-FLL}]
  Suppose that
  $(\mathcal{X}_{n_j}(\cdot),\mathcal{H}_{n_j}(\cdot)) \underset{j\to
    \infty}{\Longrightarrow} (Y(\cdot),Z(\cdot))$ along a subsequence
  $(n_j)_{j\geq 1}$. By Proposition \ref{prop:XHtight} the process
  $(Y(\cdot),Z(\cdot))$ must be continuous, and it will be enough to
  show that $(Y(\cdot),Z(\cdot))$ has the same finite dimensional
  distributions as $(\mathfrak{X}(\cdot),\mathfrak{H}(\cdot))$.  To
  this end, note that it follows from Theorem \ref{th:JWLGST} and the
  convergence assumption that for any bounded continuous function
  $\phi:\R^k\times\R_+^k \to \R$, any $\l_1,\l_2,\ldots, \l_k > 0$,
  and independent random variables
  $\gamma_1, \gamma_2,\ldots,\gamma_k$ with
  $\gamma_i\sim \text{Exp}(\lambda_i)$, $1\le i\le k$, which are also
  independent from $(Y(\cdot),Z(\cdot))$ and
  $(\mathfrak{X}(\cdot),\mathfrak{H}(\cdot))$,
\[E\left[ \phi\left( Y(\gamma_1), \ldots, Y(\gamma_k), Z(\gamma_1), \ldots, Z(\gamma_k)\right) \right]=E\left[ \phi\left( \mathfrak{X}(\gamma_1), \ldots, \mathfrak{X}(\gamma_k) , \mathfrak{H}(\gamma_1), \ldots, \mathfrak{H}(\gamma_k) \right)  \right],  
\]
or, equivalently, $\forall \l_1,\l_2,\ldots,\l_k > 0$
\begin{align*}
 &\int_{[0,\infty)^k} \left( \prod_{i=1}^k \l_i e^{-\l_i t_i} \right) E\left[ \phi\left( \mathfrak{X}(t_1), \ldots, \mathfrak{X}(t_k), \mathfrak{H}(t_1), \ldots, \mathfrak{H}(t_k)\right) \right] \, dt_1 \ldots dt_k \\
 &\qquad = \int_{[0,\infty)^k} \left( \prod_{i=1}^k \l_i e^{-\l_i t_i} \right) E\left[ \phi\left( Y(t_1), \ldots, Y(t_k),Z(t_1), \ldots, Z(t_k) \right) \right] \, dt_1 \ldots dt_k.
\end{align*}
That is, for any bounded continuous function $\phi$, the functions
\begin{align*}
  (t_1,\ldots,t_k)&\mapsto E\left[ \phi\left( \mathfrak{X}(t_1),\ldots,\mathfrak{X}(t_k),\mathfrak{H}(t_1),\ldots,\mathfrak{H}(t_k) \right) \right]\quad \text{and}
                 \\  
(t_1,\ldots,t_k)&\mapsto E\left[ \phi\left(Y(t_1),\ldots,Y(t_k),Z(t_1),\ldots,Z(t_k) \right) \right]
\end{align*}
(which are also bounded and continuous) have the same $k$-dimensional Laplace transforms and are thus equal for all $(t_1,\ldots,t_k) \in [0,\infty)^k$. This is enough to conclude that $(Y(\cdot),Z(\cdot))$ has the same finite dimensional distributions as $(\mathfrak{X}(\cdot),\mathfrak{H}(\cdot))$, and since this is true for any subsequential limit $(Y(\cdot),Z(\cdot))$ of $(\mathcal{X}_n(\cdot),\mathcal{H}_n(\cdot))$ we can conclude that indeed $(\mathcal{X}_n(\cdot),\mathcal{H}_n(\cdot)) \underset{n\to\infty}{\Longrightarrow} (\mathfrak{X}(\cdot),\mathfrak{H}(\cdot))$. 
\end{proof}

\section{Joint generalized Ray-Knight Theorems (GRKTs)}\label{sec:JRK}

In this section we will prove Theorem \ref{thm:JRK-TSAW}.  Before
beginning the proof, we note that this is essentially a
``multi-curve'' extension of the Ray-Knight type theorems proved for
the TSAW  in \cite[Theorem 1]{tTSAW} and a strengthening of
the joint GRKT stated but not proved in \cite{TW98}.

We stated the joint GRKT for the rescaled local time processes so
that it is convenient for our applications. In fact, it is the
directed edge local times processes
${\cal E}^+_{\tau_{k,m}}(\ell),\ \ell\ge k$, that we are going to use
to prove the joint GRKT, since they are Markov chains whose
probability to jump from $n$ to $n+x$, $x\in\Z$, is very close for
$n\gg 1$ to a symmetric distribution $(\pi(x))_{x\in\Z}$ with Gaussian
tails, more precisely, $\pi(x)=C\lambda^{x^2}$ for some $C>0$.
The transition probabilities of these Markov chains can be
conveniently written in terms of generalized Polya's urn processes for
which $\pi$ is the invariant distribution. The parameter $\sigma$ in
\eqref{sigma} is simply the standard deviation of $\pi$. We refer the
reader to Appendix~\ref{sec:urn} for a careful account of the above
mentioned connections and related facts.

The following proposition is all we need from Appendix~\ref{sec:urn}
to proceed with the proof of Theorem~\ref{thm:JRK-TSAW}. It says that
if we consider the joint distribution of the processes
$\left( \mathcal{E}^+_{\tau_{k,m_1}}, 
  \ldots, \mathcal{E}^+_{\tau_{k,m_N}} \right)$, then as long as the
processes remain far apart, their increments are approximately
distributed like the i.i.d.\ copies from distribution $\pi$.  The
exponential bound on the total variation distance given below will
ensure that the coupling of the increments of these processes to the
i.i.d.\ copies of $\pi$ will not break down for the required length of
time with probability tending to 1.

\begin{proposition}\label{coup}
 There are constants $C_2,C_3>0$ such that the following holds. For any fixed $k \in \Z$, $0\leq m_1 < m_2 < \cdots < m_N$, $\ell \neq 0$ and any $0\leq n_1 < n_2 < \cdots < n_N$ (if $\ell > 0$ then we need $n_1 > 0$) we have 
 \begin{align*}
 &\sum_{k_1,k_2,\ldots,k_N \in \Z}
 \left|  P\left( \mathcal{E}^+_{\tau_{k,m_i}}(\ell) 
 = n_i+k_i, \, i\leq N \mid \mathcal{E}^+_{\tau_{k,m_i}}(\ell-1) = n_i, \, i\leq N \right) - \prod_{i=1}^N \pi(k_i) \right| \\
&\qquad \leq C_3 \sum_{i=1}^N e^{-C_2(n_i-n_{i-1})}. 
\end{align*}
 \end{proposition}

We first note that Theorem \ref{thm:JRK-TSAW} holds for a single
curve, that is, for $N=1$.
\begin{theorem}[Theorem 1 in \cite{tTSAW}]\label{thm:marginalRK}
 For any fixed $(x,h) \in \R\times [0,\infty)$, the joint distribution of $\left( \Lambda^n_{x,h}(\cdot), \mu^{n,-}_{x,h}, \mu^{n,+}_{x,h} \right)$ converges as $n\to \infty$ to the distribution of $(\Lambda_{x,h}(\cdot), \mathfrak{m}^-_{x,h},\mathfrak{m}^+_{x,h})$. 
\end{theorem}
Theorem~\ref{thm:marginalRK} is somewhat different from
\cite[Theorem 1]{tTSAW}, so we will outline how the result stated
follows from \cite[Theorem 1]{tTSAW}. First of all, since
\begin{equation}\label{ltc}
 L(n,k) = \mathcal{E}_n^+(k) + \mathcal{E}_n^+(k-1)-\ind{X_0 < k < X_n} + \ind{X_n \leq k \leq X_0},
\end{equation}
it's enough to prove the convergence for directed edge local times, i.e., it suffices to show that 
\begin{equation}\label{UnRK}
 \left(\mathcal{E}^+_{\tau^n_{x,h}}(\fl{\boldsymbol{\cdot}\,n})/(\sigma \sqrt{n}), \mu^{n,-}_{x,h}, \mu^{n,+}_{x,h} \right)
 \underset{n\to\infty}{\Longrightarrow} (\Lambda_{x,h}(\cdot), \mathfrak{m}^-_{x,h},\mathfrak{m}^+_{x,h}).
\end{equation}
Note that we divide by $\sigma$ instead of $2\sigma$ (compare with
\eqref{dRKpath}) to account for using {\em directed} edge local
times. The remaining difference with the process in \cite{tTSAW} is that instead of the times $\tau_{k,m}$, \cite{tTSAW} uses the times $T^*_{k,m}$ of the $(m+1)$-th visit to $k$ from $k+1$ (or from $k-1$).
However, this difference affects only the starting point of the Markov
chains of directed edges, $\mathcal{E}^+_{\tau^n_{x,h}}(\cdot)$ or
$\mathcal{E}^+_{T^*_{\fl{xn},\fl{\sigma h\sqrt{n}}}}(\cdot)$, to the
left and right of $\fl{xn}$. Moreover, at the stopping time
$\tau^n_{x,h}=\tau_{\fl{xn},\fl{2\sigma h\sqrt{n}}}$ the walk will
have made approximately half of its steps to $\fl{xn}$ from the right
and half from the left in the sense that
\begin{equation}\label{Lnxh-initial}
 \lim_{n\to\infty} \mathcal{E}^+_{\tau^n_{x,h}}(\fl{xn})/\sqrt{n} = \s h, \qquad P\text{-a.s.}
\end{equation}
while the definition of $T^*_{k,m}$ and of the directed edge
local times imply that\linebreak
$\mathcal{E}^+_{T^*_{\fl{xn},\fl{\s h \sqrt{n}} }}(\fl{xn}) = \fl{\s h
  \sqrt{n}}+\ind{x\geq 0}$. Thus, there is essentially no
difference between the processes in \eqref{UnRK} and in \cite[Theorem
1]{tTSAW}.

The local time curve $\Lambda^n_{x,h}(\cdot)$ for the random walk and
the corresponding local time curve $\Lambda_{x,h}(\cdot)$ for the
TSRM consist of a ``forward'' and ``backward'' curve. For
convenience of notation we will use $\Lambda^{n,+}_{x,h}(\cdot)$ and
$\Lambda^{n,-}_{x,h}$ to denote the process $\Lambda^n_{x,h}(\cdot)$
restricted to $[x,\infty)$ and $(-\infty,x]$, respectively. Similarly,
we will let $\Lambda^+_{x,h}(\cdot)$ and $\Lambda^-_{x,h}(\cdot)$
denote the process $\Lambda_{x,h}(\cdot)$ restricted to $[x,\infty)$
and $(-\infty,x]$, respectively.  For a fixed $(x,h)$, the forward and
backward discrete curves $\Lambda^{n,+}_{x,h}$ and
$\Lambda^{n,-}_{x,h}$ are asymptotically independent. Indeed, since
the directed edge local times to the right and left, respectively, are
both Markov chains, the dependence is only through the initial
conditions, but as noted above the initial conditions of the forward
and backward curves are deterministic in the limit (see
\eqref{Lnxh-initial}).  Therefore, Theorem \ref{thm:marginalRK} can be
proved by handling the forward and backward curves separately. That
is, we can prove Theorem \ref{thm:marginalRK} by proving that
$(\Lambda^{n,+}_{x,h}(\cdot), \mu_{x,h}^{n,+}) \Rightarrow
(\Lambda^+_{x,h}(\cdot), \mathfrak{m}_{x,h}^+)$ and
$(\Lambda^{n,-}_{x,h}(\cdot), \mu_{x,h}^{n,-}) \Rightarrow
(\Lambda^-_{x,h}(\cdot), \mathfrak{m}_{x,h}^-)$.
\smallskip

However, to prove the joint GRKT, Theorem \ref{thm:JRK-TSAW}, a
difficulty arises trying to handle the forward and backward curves for
different space-local time starting points. More precisely, given two
points $(x_1,h_1)$ and $(x_2,h_2)$ with $x_1<x_2$,
  it is not difficult to study the joint
  distribution of the forward curves
  $(\Lambda^{n,+}_{x_1,h_1}(\cdot), \Lambda^{n,+}_{x_2,h_2}(\cdot))$
  or of the backward curves
  $(\Lambda^{n,-}_{x_1,h_1}(\cdot), \Lambda^{n,-}_{x_2,h_2}(\cdot))$,
  but the joint distribution of $(\Lambda^{n,+}_{x_1,h_1}(\cdot), \Lambda^{n,-}_{x_2,h_2}(\cdot))$
  turns out to be much more complicated to work with.  In fact, for the limiting curves
    $\Lambda_{x_i,h_i}(\cdot)$ corresponding to the TSRM (i.e., in the
    Brownian web), it is known that the forward and backward curves
    ``reflect off one another'' in a precise sense \cite{STW00}, but
    it seems quite non-trivial to prove that
    this reflection property holds asymptotically for the families of
    discrete forward and backward curves
    $\Lambda_{x_i,h_i}^{n,\pm}(\cdot)$.   Thankfully, as we will
    show in Corollary~\ref{cor:FJRK-JRK}, it will
    be enough to only consider the joint distributions of the forward
    curves (or the backward curves) alone. This is due to a discrete version of the duality which we exploit in the proof of Corollary~\ref{cor:FJRK-JRK}.  The following
    result gives forward joint GRKT that will be needed. It can be
    equivalently stated for the backward curves.

\begin{theorem}[Joint forward GRKT]\label{thm:JRK-forward}
  For any $N\in\N$ and any choice of
  $(x_i,h_i)\in \R \times \R_+$, $1\le i\le N$, the
  joint distribution of the processes
  $\Lambda^{n,+}_{x_i,h_i}(\cdot)$, the rescaled
  endpoints
  $\mu^{n,+}_{(x_i h_i)} $,
  and merging points
  $\mu^{n,+}_{(x_i,h_i),(x_j,h_j)}$, $1\le i\leq N, \, i<j\le N$, converges as $n\to\infty$ to the joint distribution of
  $\Lambda^+_{x_i,h_i}(\cdot)$,
  $\mathfrak{m}^+_{x_i,h_i}$,
  and
  $\mathfrak{m}^+_{(x_i,h_i),(x_j,h_j)} $, $1\le i\leq N, \, i<j\le N$.
\end{theorem}

\begin{proof}
  Due to the relationship between local times of directed edges and
  local times of sites in \eqref{ltc}, it suffices to prove a
  corresponding joint diffusion limit for directed edge local
  times.
  That is, we will prove the joint convergence with the site
    local time profiles $\Lambda_{x_i,h_i}^{n,+}(\cdot)$ replaced by
    the directed edge local time processes
    $\left(\mathcal{E}^+_{\tau_{x_i,h_i}^n}(\fl{ y n})/(\s\sqrt{n})
    \right)_{y\geq x_i}$.  Moreover, by the Markov property for
    the directed edge local times there is no loss of generality in
  assuming that $x_i=x$ for all $i$.  Indeed, by induction
    on the number of points, if
    $x_1\leq x_2 \leq \cdots \leq x_{k-1} \leq x_k$, then by
    conditioning on the values of
    $h_{i,k}^n = \mathcal{E}^+_{\tau_{x_i,h_i}^n}(\fl{x_k
      n})/(\s\sqrt{n})$ it's enough to prove the joint convergence of
    the curves started from the pairs
    $(x_k,h_{1,k}^n), (x_k,h_{2,k}^n),\ldots,(x_k,h_{k-1,k}^n),
    (x_k,h_k)$.

    Fix an $x\in\R$. Details of the proof depend on the sign of
    $x$ but the difference between the three cases ($x<0,\ x=0,\ x>0$)
    is very minimal.  For that reason, we shall only consider the case
    $x<0$, in which we have reflection on $[x,0)$ and absorption on
    $[0,\infty)$. The other two cases have no reflection interval and
    are simpler. The proof proceeds by induction on $N$.  The case
    $N=1$ is contained in Theorem~\ref{thm:marginalRK}. See also
    \eqref{UnRK} and discussion below it.

  We shall first treat the case $N=2$ in detail. Fix $x<0$,
  $\epsilon\in(0,1/100)$, and $0\le h_1<h_2$.  We start by lining up
  several facts about two sequences of rescaled reflected coalescing
  random walks with independent integer-valued increments
  $\zeta_i(j),\ j\in\N,\ i=1,2$, with mean $0$, variance
  $\sigma^2\in(0,\infty)$, and a finite third moment. For $i=1,2$ let
\[\lim_{n\to\infty}Y^n_i(0)=h_i\ \text{a.s},\quad
  Y^n_i\left(\frac{j}{n}\right)=\left|Y^n_i\left(\frac{j-1}{n}\right)+\frac{\zeta_i(j)}{\sigma\sqrt{n}}\right|,\quad
  j\ge 1.\] At each step we first connect
$Y^n_i\left(\frac{j-1}{n}\right)$ and
$Y^n_i\left(\frac{j-1}{n}\right)+\frac{\zeta_i(j)}{\sigma\sqrt{n}}$
and then reflect at $0$ if the second point is negative. We get a
continuous piecewise linear curve $Y^n_i(y),\ y\ge 0$, which hits $0$
in the interval $((j-1)/n, j/n)$ if the expression inside the absolute
value above is negative. We also define $\t{Y}^n_2$, to be the walk
which starts from $Y^n_2(0)-\fl{n^\epsilon}/(\sigma\sqrt{n})$ and uses
the same increments as $Y^n_2$. Note that
$\t{Y}^n_2\left(\frac{j}{n}\right)=Y^n_2\left(\frac{j}{n}\right)-\fl{n^\epsilon}/(\sigma\sqrt{n})$
up to (not including) the first time the right hand side falls below
$0$. We produce $\t{Y}^n_2(y),\ y\ge 0$, by the same kind of linear
interpolation as above and define the hitting times of $0$ after time
$|x|$ and the crossing time of $Y^n_1$ and $\t{Y}^n_2$ respectively as
follows:
\begin{align*}
  &\nu^n_1=\inf\{y\ge |x|: Y^n_1(y)=0\},\ \ \t{\nu}^n_2=\inf\{y\ge |x|: \t{Y}^n_2(y)=0\},\\ & \t{c}^n_{1,2}=\inf\{y\ge 0: \t{Y}^n_2(y)=Y^n_1(y)\}\ \left(=\inf\{y\ge 0: Y^n_2(y)-Y^n_1(y)=\fl{n^\epsilon}/(\sigma\sqrt{n})\}\right).
\end{align*}
\begin{itemize}
\item[(I)] By independence of $Y^n_1(\cdot)$ and $\t{Y}^n_2(\cdot)$,
  convergence of one-dimensional processes (see \cite{NP}\footnote{In
    \cite{NP} the authors first construct a discrete reflected random
    walk and only then take the linear interpolation. Our definition
    naturally keeps track of all reflection times of interpolated
    walks. The difference in the definitions does not affect the
    convergence.}), and the continuous mapping theorem,
  \[(Y^n_1(\cdot),\t{Y}^n_2(\cdot),\nu^n_1,\t{\nu}^n_2,\t{c}^n_{1,2})\Longrightarrow
  (|W_1(\cdot)|,|W_2(\cdot)|,\nu_1,\nu_2,c_{1,2})\quad\text{ as $n\to\infty$},\]
  where the independent Brownian motions $W_1(\cdot)$ and $W_2(\cdot)$
  start at $h_1$ and $h_2$ respectively and the last three stopping times
  are defined in the same way as for $Y^n_1(\cdot)$ and
  $\t{Y}^n_2(\cdot)$. The continuous mapping theorem is applicable,
  since the set of discontinuities of the last three functions on the path
  space of a standard two dimensional Brownian motion has measure $0$.
\item[(II)] We define a mapping $\phi$ which takes a pair of paths
  $(f_1,f_2)\in C([0,\infty))^2,\ f_1,f_2\ge 0$,
  $f_1(0)\le f_2(0)$, and returns a pair of coalescing paths
  $\phi(f_1,f_2)$ as follows. Let
  $c(f_1,f_2):=\inf\{y\ge 0:f_1(y)=f_2(y)\}\in[0,\infty]$ and
  $\nu_0(f_i)=\inf\{y\ge|x|:f_i(y)=0\}\in[|x|,\infty]$. We shall refer
  to $\nu_0$ as a ``marked time''.  Set \[\phi(f_1,f_2)(y):=
    \begin{cases}
      (f_1(y),f_2(y)),&\text{if }y\le c(f_1,f_2);\\
      (f_1(y),f_1(y)),&\text{if }y>c(f_1,f_2)\ \text{and
      }c(f_1,f_2)\le \nu_0(f_1);\\ (f_2(y),f_2(y)),&\text{if
      }y>c(f_1,f_2)\ \text{and }c(f_1,f_2)>\nu_0(f_1).
    \end{cases}
  \] In words, at the time when the two paths meet we keep only the
  lower path if they meet before or at the marked time of the lower
  path and we keep only the upper path if they meet after the marked
  time of the lower path. This mapping is a.s.\ continuous for the
  pair of independent reflected Brownian motions. Hence, just as
  above, we can conclude the convergence of a pair of rescaled
  coalescing reflected random walks together with their marked and
  coalescence times to a pair of coalescing reflected Brownian motions
  together with their marked and coalescence times.
\end{itemize}

From now on we assume that $\zeta_i(j)\sim\pi$, $i=1,2$,
$j\in\N$. Set
\begin{equation}\label{S}
 S^n_i(j)=\mathcal{E}^+_{\tau^n_{x,h_i}}(j+\fl{xn}),\ j\ge 0; \quad W^n_i(y) = \frac{S^n_i(yn)}{\sigma\sqrt{n}},\quad yn\in \N_0, \quad i=1,2,
\end{equation}
and linearly interpolated in between so that $W^n_i\in
C([0,\infty))$. For convenience we set $S^n_0(j)=0$ for all $j\ge 0$.
Note that $W^n_i(0)\to h_i$, $i=1,2$, a.s.\ as $n\to\infty$, see
\eqref{Lnxh-initial}.

Recall that $\epsilon\in(0,1/100)$ and define for $i=1,2$
\[\omega^{n,4\epsilon}_i=\min\{j\ge \fl{|x|n}: S^n_i(j)\le
n^{4\epsilon}\}, \quad 
\omega^{n,\epsilon}_{1,2}=\min\{j\ge 0: S^n_2(j)-S^n_1(j)\le
n^\epsilon\}.\] We consider the pair $(S^n_1,S^n_2)$ and enlarge the probability
space as necessary to construct a pair of crossing marked and reflected
Markov chains $(R^n_1,\t{R}^n_2)$ as follows. Let $R^n_1(0)=S^n_1(0)$ and
$\t{R}^n_2(0)=S^n_2(0)-\fl{n^\epsilon}$.

Up until $\omega^{n,\epsilon}_{1,2}\wedge\omega^{n,4\epsilon}_1$ we
set $R^n_1(j)=S^n_1(j)$, couple the increments of $S^n_2$ with
independent increments $\zeta_2(j)$ at each step $j$, and
set $\t{R}^n_2(j)=\t{R}^n_2(0)+\sum_{k=1}^j \zeta_2(j)$. Given the
initial data, the coordinate processes of $(R^n_1,\t{R}^n_2)$ are
independent, and
$(R^n_1(j),\t{R}^n_2(j))=(S^n_1(j),S^n_2(j)-\fl{n^\epsilon})$ for
$j<\omega^{n,\epsilon}_{1,2}\wedge\omega^{n,4\epsilon}_1$, provided
that the coupling does not break down.
\begin{enumerate}[label=(\roman*)]
\item If $\omega^{n,\epsilon}_{1,2}<\omega^{n,4\epsilon}_1$ then
  \begin{enumerate}[label=(\alph*)]
  \item at time $\omega^{n,\epsilon}_{1,2}$ we decouple $\t{R}^n_2$
    from $S^n_2$ and simply let $\t{R}^n_2$ use increments
    $\zeta_2(j)$, i.e., $\t{R}^n_2(j)=|\t{R}^n_2(j-1)+\zeta_2(j)|$ for all $j>\omega^{n,\epsilon}_{1,2}$;
  \item $R^n_1$ continues to follow $S^n_1$ up to time
    $\omega^{n,4\epsilon}_1$, after which it uses independent
    increments $\zeta_1(\cdot)$ and follows a reflected
    marked random walk.
  \end{enumerate}
\item If $\omega^{n,\epsilon}_{1,2}\ge \omega^{n,4\epsilon}_1$ then
  \begin{enumerate}[label=(\alph*)]
  \item at time $\omega^{n,4\epsilon}_1$ we decouple $R^n_1$ from
    $S^n_1$, $R^n_1$ continues to use independent increments $\zeta_1(\cdot)$ and
    follows a reflected marked random walk;
  \item we continue to couple the increments of $\t{R}^n_2$ and $S^n_2$ up to time $\omega^{n,4\epsilon}_2$, starting from which we let $\t{R}^n_2$ use  increments $\zeta_2(\cdot)$ and continue as a reflected marked random walk.
  \end{enumerate} 
\end{enumerate}
We agree that if any of the couplings above breaks down before the
specified decoupling time, we shall just use the corresponding
independent $\pi$-distributed increments for the rest of the
time. Since a single curve gets absorbed within $n^{1+\epsilon}$ units
of time with probability tending to 1, Lemma~\ref{notdecoup} implies
that the above couplings will not break down before the specified
decoupling time with probability tending to 1 as $n\to\infty$.

Diffusively rescaling the process $(R^n_1,\t{R}^n_2)$ and ``connecting
the dots'' in the same way as before we get a sequence of pairs of
independent reflected marked processes which converge in distribution
together with their rescaled marked and crossing times to a pair of
independent marked reflected Brownian motions and their marked and
crossing times.  Indeed, by the result for $N=1$ and fact (I), each
sequence of the rescaled coordinate processes together with their
rescaled marked times converges to a reflected Brownian motion and its
marked time. Since the increments of the coordinate processes
are independent, the claimed convergence follows. By (II), we get the
convergence of rescaled marked coalescing reflected processes
$(R^n_1,\t{R}^n_2)$ and sequences of their rescaled marked and
coalescence times to a pair of coalescing reflected Brownian motions
together with their marked and coalescence times.

Our next task is to sort out the marked and coalescence times and
conclude the convergence of $(W^n_1(\cdot),W^n_2(\cdot))$ together
with their coalescence and absorption times to a pair of coalescing
independent reflected/absorbed Brownian motions together with their
coalescence and absorption times.

We shall use results from Appendix~\ref{sec:aux} to show that the
coalescence and absorption times we are interested in are close (in
the sense that after scaling, with probability tending to 1, the difference is
negligible) to the crossing and marked times for which
the convergence has been shown above.

Let us start with scenario (i) (coalescence before the marked time for
$R^n_1$). There are two cases: (i1)
$S^n_2(\omega^{n,\epsilon}_{1,2})-\fl{n^\epsilon}\le n^{4\epsilon}$
and (i2)
$S^n_2(\omega^{n,\epsilon}_{1,2})-\fl{n^\epsilon}> n^{4\epsilon}$.
In case (i1), with probability tending to 1, by the result for a
single curve, \cite[(4.69)]{tTSAW}, $S^n_2$ will hit zero (but not
necessarily gets absorbed) in at most $\ceil{n^{28\epsilon}}$
additional steps. This means that $S^n_1$ and $S^n_2$ must have
coalesced within the same number of steps, which is negligible after
division by $n$ as $n\to\infty$. This shows the convergence of the
coalescence times in this case.

In case (i2), we have that
$\t{R}^n_2(\omega^{n,\epsilon}_{1,2}),\,R^n_1(\omega^{n,\epsilon}_{1,2})>
n^{4\epsilon}$. It follows from Corollary~\ref{close} (with
$b= \fl{n^\epsilon}$) that with probability tending to 1, the crossing
time of $R^n_1$ and $\t{R}^n_2$ and the actual coalescence time of
$S^n_1$ and $S^n_2$ will differ by less than $n^{7\epsilon}$, which is
negligible after division by $n$.

In both cases, again by the result for a single curve,
\cite[(4.69)]{tTSAW}, with probability tending to 1, the common
absorption time of the coalesced processes differs from the marked
time of $R^n_1$ by no more than $n^{28\epsilon}$, which is negligible
after division by $n$.

Under scenario (ii) ($S^n_1$ gets within $n^{4\epsilon}$ from $0$
after time $\fl{|x|n}$ and before coalescence), again by
\cite[(4.69)]{tTSAW}, with probability tending to 1 the marked time of
$R^n_1$ will be within $n^{28\epsilon}$ from the absorption time of
$S^n_1$.  This difference is negligible after division by $n$. The
crossing time of $R^n_1$ and $\t{R}^n_2$ will be irrelevant for our
purposes, and the absorption time of $S^n_2$ will be within
$n^{28\epsilon}$ of the marked time of $\t{R}^n_2$ with probability
tending to 1.

From the above we conclude the desired convergence for the case $N=2$.
Suppose that we have the claimed result for all $1\le k\le N-1$ and
any choice of $x<0$ and $0\le h_1<h_2<\ldots<h_{N-1}$. Let
$h_N>h_{N-1}$ and consider the walks $(S^n_i(\cdot) )_{i\leq N}$ and
their rescaled versions $(W^n_i(\cdot) )_{i\leq N}$.  Let
$\omega^{n,\epsilon}_{N-1,N}=\min\{j\ge 0: S^n_N(j)-S^n_{N-1}(j)\le
n^\epsilon\}$. We shall take the pair $(S^n_{N-1}, S^n_N)$ and
construct $(R^n_{N-1}, \t{R}^n_N)$ using the same algorithm as for
$(R^n_1,\t{R}^n_2)$ above. By construction, given the initial
  data, the reflected marked random walk $\t{R}^n_N$ is independent
from all $(S^n_i(\cdot) )_{i\leq N-1}$. By the same argument as for
$N=2$, we can apply (II) to rescaled coalescing marked reflected
processes $(R^n_{N-1}, \t{R}^n_N)$ and in the same way as for $N=2$
deduce the convergence of $(W^n_{N-1},W^n_N)$ together with their
coalescence and absorption times to a pair of coalescing independent
reflected absorbed Brownian motions starting from $(h_{N-1},h_N)$
together with their coalescence and absorption times. The induction
hypothesis, the convergence of the initial data to deterministic
  values, and the fact that the increments of $\t{R}^n_N$ are
independent from all $(S^n_i(\cdot) )_{i\leq N-1}$ yield the statement
of the theorem for $x<0$.
 \end{proof}

Next, we show how Theorems \ref{thm:marginalRK} and \ref{thm:JRK-forward} can be combined to give the joint convergence of finitely many (forward and backward) Ray-Knight curves. 

  \begin{corollary}\label{cor:FJRK-JRK}
The joint GRKT, Theorem~\ref{thm:JRK-TSAW}, follows from the forward joint GRKT, Theorem \ref{thm:JRK-forward}, and the GRKT for a single curve, Theorem \ref{thm:marginalRK}. 
\end{corollary}

\begin{proof}
 Given any finite set of points $(x_1,h_1),\ldots,(x_N,h_N)$, it follows from Theorem \ref{thm:marginalRK} that the joint distribution of the curves $(\Lambda_{x_i,h_i}^n(\cdot))_{i\leq N}$ and the endpoints $\{\mu_{x_i,h_i}^{n,*} \}_{i\leq N,\, * \in \{+,-\} }$ is tight. Moreover, any subsequential limit must have the property that the limiting curves $\Lambda_{x_i,h_i}(\cdot)$ are continuous and the subsequential limits of the endpoints $\mu_{x_i,h_i}^{n,*}$  are equal to the endpoints  of the curves $\Lambda_{x_i,h_i}(\cdot)$. Furthermore, it follows from the forward (resp.\ backward) joint GRKT that the subsequential limits of the merge points $\mu_{(x_i,h_i),(x_j,h_j)}^{n,+}$ (resp.\ $\mu_{(x_i,h_i),(x_j,h_j)}^{n,-}$) must be the merge points of the curves $\Lambda_{x_i,h_i}(\cdot)$ and $\Lambda_{x_j,h_j}(\cdot)$.

 Therefore, it remains only to show that the subsequential limit of
 the joint distribution of the curves is unique.  To this end, for
 each $i\leq N$, $m_i\in\N$, let $(y_{i,j})_{j\leq m_i}$ be a
 finite collection of spatial points. We need to show that the
 limiting distribution of the collection of random variables
 $\Lambda_{x_i,h_i}^n(y_{i,j})$, $i\leq N$, $j\leq m_i$, is determined
 by the forward (or backward) joint GRKT.  If 
 $y_{i,j} \geq x_i$ then it is part of the forward curve
 $\Lambda_{x_i,h_i}^{n,+}(\cdot)$ and the joint convergence of these
 finite dimensional distributions is covered by Theorem
 \ref{thm:JRK-forward}. However, the following lemma shows that the finite dimensional distributions of
 the discrete backward curves $\Lambda_{x_i,h_i}^{n,-}(\cdot)$ are
 determined by the finite dimensional distributions
 of the discrete forward curves.

\begin{lemma}\label{lem:weave}
 For any $y\neq x$ and $h,h'\geq 0$
 we have that if $n$ is large enough so that $\fl{xn} \neq \fl{yn}$ then 
\begin{equation}\label{interweave}
 \{\Lambda_{x,h}^n(y) \leq h' \} = \{ \Lambda_{x,h}^n(x) \leq \Lambda_{y,h'}^n(x) \}.
\end{equation}
\end{lemma}

\begin{remark}
 Note that since $\Lambda_{x,h}^n(x) = \frac{\fl{2\s h\sqrt{n}}+1}{2\s \sqrt{n}} \to h$ as $n\to\infty$, \eqref{interweave} is the discrete analog of the property that the events $\{\Lambda_{x,h}(y) \leq h' \}$ and $ \{ h \leq \Lambda_{y,h'}(x) \}$ differ by a null set (see Corollary \ref{cor:sandwich} for a more precise statement). 
\end{remark}

Before giving the proof of Lemma \ref{lem:weave} we shall use it to justify our last claim. 
If $y< x$, then the event on the left side of \eqref{interweave} involves  
the backward curve from $(x,h)$ while the event on the right side of \eqref{interweave} only involves the forward curve from $(y,h')$. 

Now, suppose the spatial points $(y_{i,j})_{i\leq N, j\leq m_i}$ are ordered so that $y_{i,j} < x_i$ for $j\leq m_i'$ and $y_{i,j} \geq x_i$ for $m_i' < j \leq m_i$, 
and for each $i\leq N$ and $j\leq m_i$ fix $b_{i,j}>0$. 
Then it follows from Lemma \ref{lem:weave} 
that for $n$ sufficiently large we have
\begin{align*}
&P\left( \Lambda_{x_i,h_i}^n(y_{i,j}) \leq b_{i,j}, \quad \forall i\leq N, j\leq m_i \right)\\
&= P\left( \bigcap_{i\leq N} \left( \bigcap_{j\leq m_i'} \left\{ \Lambda_{x_i,h_i}^n(x_i) \leq \Lambda_{y_{i,j},b_{i,j}}^n(x_i) \right\} 
\cap 
\bigcap_{m_i' < j \leq m_i} \left\{ \Lambda_{x_i,h_i}^n(y_{i,j}) \leq b_{i,j} \right\} \right) \right).
\end{align*}
Since the events in the probability on the right depend only on the finite dimensional distributions of the forward curves it then follows from Theorem \ref{thm:JRK-forward} that 
\begin{align}
 \lim_{n\to\infty} &P\left( \Lambda_{x_i,h_i}^n(y_{i,j}) \leq b_{i,j}, \quad \forall i\leq N, j\leq m_i \right) \nonumber \\
 &= P\left( \bigcap_{i\leq N} \left( \bigcap_{j\leq m_i'} \left\{ h_i < \Lambda_{y_{i,j},b_{i,j}}(x_i) \right\} 
\cap 
\bigcap_{m_i' < j \leq m_i} \left\{ \Lambda_{x_i,h_i}(y_{i,j}) \leq b_{i,j} \right\} \right) \right). \label{Lnxhy-lim}
\end{align}
A few points of further explanation are due regarding the application of Theorem \ref{thm:JRK-forward} in this last equality.

Let $m_i' < j \leq m_i$. Since  $\Lambda_{x_i,h_i}(\cdot)$ is a reflected/absorbed Brownian motion and $b_{i,j} > 0$, it follows that   $\Lambda_{x_i,h_i}(y_{i,j})$ doesn't have an atom at $b_{i,j}$.  
Thus Theorem \ref{thm:JRK-forward} justifies replacing the events $\{\Lambda_{x_i,h_i}^n(y_{i,j}) \leq b_{i,j} \}$ by the corresponding events $\{\Lambda_{x_i,h_i}(y_{i,j}) \leq b_{i,j} \}$ in the limit.

Let $j\leq m_i'$. To justify replacing the events $\{ \Lambda_{x_i,h_i}^n(x_i) \leq \Lambda_{y_{i,j},b_{i,j}}^n(x_i) \}$ with the corresponding events $\{ h_i < \Lambda_{y_{i,j},b_{i,j}}(x_i)\}$ we need to consider two cases. 
 \begin{description}
  \item[Case $h_i>0$ or $x_i \leq 0$] this case is justified by noting that $\Lambda_{x_i,h_i}^n(x_i) = \frac{\fl{2\s h_i \sqrt{n}}+1}{2\s h_i \sqrt{n}} \to h_i$ as $n\to \infty$ and that since $\Lambda_{y_{i,j},b_{i,j}}(\cdot)$  is a reflected/absorbed Brownian motion then the distribution of $\Lambda_{y_{i,j},b_{i,j}}(x_i)$ doesn't have an atom at $h_i$ if either $h_i>0$ or if $x_i \leq 0$. 
  \item[Case $h_i=0$ and $x_i>0$] in this case, first of all note that
  \[
   \left\{ \Lambda_{x_i,0}^n(x_i) \leq \Lambda_{y_{i,j},b_{i,j}}^n(x_i) \right\} 
   = \{ \tau^n_{y_{i,j},b_{i,j}} > \tau^n_{x_i,0} \}
   = \left\{ n\mu^{n,+}_{y_{i,j},b_{i,j}} \geq \fl{x_i n} \right\}.
  \]
 Since Theorem \ref{thm:JRK-forward} implies that $\mu^{n,+}_{y_{i,j},b_{i,j}} \Rightarrow \mathfrak{m}^+_{y_{i,j},b_{i,j}}$, and since $\mathfrak{m}^+_{y_{i,j},b_{i,j}}$ has a continuous distribution on $(0,\infty)$ it follows that we can replace the event $\{ \Lambda_{x_i,0}^n(x_i) \leq \Lambda_{y_{i,j},b_{i,j}}^n(x_i)\}$  with the event $\{ \mathfrak{m}^+_{y_{i,j},b_{i,j}} > x \}$ in the limit. 
 Finally, note that $\{ \mathfrak{m}^+_{y_{i,j},b_{i,j}} > x_i \} = \{\Lambda_{y_{i,j},b_{i,j}}(x_i) > 0 \}$. 
 \end{description}

Finally, since it follows from Corollary \ref{cor:sandwich} that 
\[
 \left\{ \Lambda_{x_i,h_i}(y_{i,j}) < b_{i,j} \right\} \subset \{h_i < \Lambda_{y_{i,j},b_{i,j}}(x_i)\}
 \subset \left\{ \Lambda_{x_i,h_i}(y_{i,j}) \leq b_{i,j} \right\}, 
\]
and since $P(\Lambda_{x_i,h_i}(y_{i,j}) = b_{i,j}) = 0$ when $b_{i,j}>0$ we can conclude from \eqref{Lnxhy-lim} that 
\begin{align*}
 &\lim_{n\to\infty} P\left( \Lambda_{x_i,h_i}^n(y_{i,j}) \leq b_{i,j}, \quad \forall i\leq N, j\leq m_i \right) 
= P\left( \Lambda_{x_i,h_i}(y_{i,j}) \leq b_{i,j}, \quad \forall i\leq N, j\leq m_i \right). 
\end{align*}
This completes the proof of the corollary, pending the proof of Lemma \ref{lem:weave}.
\end{proof}

\begin{proof}[Proof of Lemma \ref{lem:weave}]
 For any integers $j\neq k$ and $m\geq 0$ we claim that
 \begin{equation}\label{backpath}
  L(\tau_{k,m},j) = \min\{ \ell \geq 0: L(\tau_{j,\ell},k) \geq m+1 \}. 
 \end{equation}
Indeed, this can be justified by noting that 
\[
 L(\tau_{k,m},j) = 0 
 \iff \tau_{k,m} < \tau_{j,0}
 \iff L(\tau_{j,0},k) \geq m+1,
\]
and that for $\ell \geq 1$, 
\[
  L(\tau_{k,m},j) = \ell 
 \iff \tau_{j,\ell-1} < \tau_{k,m} < \tau_{j,\ell}
 \iff L(\tau_{j,\ell-1},k) < m+1 \leq L(\tau_{j,\ell},k). 
\]
Translating the equation \eqref{backpath} to the rescaled path $\Lambda_{x,h}^n(\cdot)$ we have
\begin{align*}
 \Lambda^n_{x,h}(y) 
 &= \frac{L(\tau_{\fl{xn},\fl{2\s h \sqrt{n}}}, \fl{yn})}{2 \s \sqrt{n}} \\
 &= \frac{1}{2\s \sqrt{n}} \min\left\{ \ell \geq 0: L(\tau_{\fl{yn},\ell}, \fl{xn}) \geq \fl{2\s h\sqrt{n}} + 1 \right\} \\
 &= \inf\left\{ b\geq 0: L(\tau_{\fl{yn},\fl{2\s b \sqrt{n}}}, \fl{xn}) \geq \fl{2\s h\sqrt{n}} + 1 \right\} \\
 &= \inf\left\{ b\geq 0: \Lambda_{y,b}^n(x) \geq \frac{\fl{2\s h \sqrt{n}}+1}{2\s \sqrt{n}} \right\} = \inf\left\{ b\geq 0: \Lambda_{y,b}^n(x) \geq \Lambda_{x,h}^n(x) \right\}.
\end{align*}
From this we see that 
\begin{equation*}
 \Lambda_{x,h}^n(y) \leq h' \iff \inf\left\{ b\geq 0: \Lambda^n_{y,b}(x) \geq \Lambda_{x,h}^n(x) \right\} \leq h' \iff \Lambda_{x,h}^n(x) \leq \Lambda_{y,h'}^n(x). 
\end{equation*}
In the last equivalence we used the fact that $b\mapsto\Lambda_{y,b}^n(x)$ is non-decreasing and right continuous.

\end{proof}

\section{Tightness}\label{sec:tight}

In this section we shall prove Proposition~\ref{prop:XHtight}. The
proof will consist of two steps. First, in
Proposition~\ref{prop:Xtight} we will prove the tightness of
$({\cal X}_n)_{n\ge 1}$ and then in Proposition~\ref{prop:Htight} will
show that the sequence $({\cal H}_n)_{n\ge 1}$ is tight (recall that $\mathcal{X}_n$ and $\mathcal{H}_n$ are defined in \eqref{XH}). The joint
tightness will immediately follow.

\begin{lemma}\label{lem:Krange}
For integers $K,n\geq 1$ let $R_{K,n}$ be the event 
 \[
  R_{K,n} = \left\{ \max_{k\leq n^{3/2}} |X_k| \leq K n \text{ and } \max_{k \in \Z} L(\fl{n^{3/2}},k) \leq 2\s K \sqrt{n} \right\}. 
 \]
For any $\eta>0$ there exists $K <\infty$ such that  
for $n$ sufficiently large, $P(R_{K,n}) > 1-\frac{\eta}{3}$. 
\end{lemma}

\begin{proof}
 First of all, 
 note that $R_{K,n} \supset R_{K,n,-} \cap R_{K,n,+}$, where 
 \begin{align*}
  R_{K,n,\pm} &= \left\{ \tau_{\pm Kn,0} > n^{3/2} \text{ and } \max_{k\in \Z} L(\tau_{\pm Kn,0},k) \leq 2\s K \sqrt{n} \right\} \\
  &= \left\{ \int_\R \Lambda_{\pm K,0}^n(x) \, dx > (1 + n^{-3/2})/(2\s) \text{ and } \sup_{x \in \R} \Lambda_{\pm K,0}^n(x) \leq K \right\},
 \end{align*}
where we used \eqref{txn-integral} in the last equality.
It follows from Theorem \ref{thm:JRK-TSAW} and Corollary \ref{cor:tau-lim}, along with symmetry considerations and scaling properties of Brownian motion, that 
\begin{align*}
 \limsup_{n\to\infty} P\left( R_{K,n}^c \right)
 &\leq 2 P\left( \int_\R \Lambda_{-K,0}(x) \, dx \leq \frac{1}{2\s} \text{ or } \sup_{x \in \R} \Lambda_{-K,0}(x) > K \right) \\
 &= 2 P\left( \int_\R \Lambda_{-1,0}(x) \, dx \leq \frac{1}{2\s K^{3/2}}  \text{ or } \sup_{x \in \R} \Lambda_{-1,0}(x) > \sqrt{K} \right). 
\end{align*}
Since the last probability above can be made arbitrarily small by taking $K$ sufficiently large, the conclusion of the lemma follows easily. 
\end{proof}

For the next lemma, to make the notation easier we introduce the following definition. 
\begin{defn}
 For a finite set $\mathbf{X} \subset \R$, the gap size of $\mathbf{X}$ is 
 \[
  \gap(\mathbf{X}) = \min\{ |x-y|: \, x,y \in \mathbf{X}, \, x\neq y \}. 
 \]
\end{defn}

We also need to introduce the following notation. Recall \eqref{tnxh} and  
for $\d>0$ and $n\geq 1$ let 
\begin{align*}
  \mathcal{T}_{\d,n} &= \left\{ \frac{ \tau^n_{x,h}}{n^{3/2}} : (x,h) \in (\delta \Z ) \times (\sqrt{\delta} \Z_+  ) \right \},  \\
 \mathcal{T}_{\d,n}^{(K)} &= \left\{ \frac{ \tau^n_{x,h}}{ n^{3/2}} : (x,h) \in (\delta \Z \cap [-K,K] ) \times (\sqrt{\delta} \Z \cap [0,K] ) \right \}, \quad\forall K\ge 1. 
\end{align*}

\begin{lemma}\label{lem:gap}
  For any fixed $K\geq 1$ and $\eta,\d>0$, there exists a
  $\d' = \delta'(K,\eta,\d) >0$ such that for $n$ sufficiently large
\[
 P\left( \gap(\mathcal{T}_{\d,n}^{(K)}) > \d' \right) > 1- \frac{\eta}{3}. 
\]
\end{lemma}


\begin{proof}
 Let $\mathcal{T}_{\delta}^{(K)} = \left\{ \mathfrak{t}_{x,h}: (x,h)\in (\delta \Z \cap [-K,K]) \times (\sqrt{\delta} \Z \cap [0,K]) \right \}$. 
 It follows from
 Corollary \ref{cor:tau-lim}
 that $\gap(\mathcal{T}_{\d,n}^{(K)})$ converges in distribution to $2\s\gap(\mathcal{T}_{\delta}^{(K)})$. 
 Since the map $(x,h) \mapsto \mathfrak{t}_{x,h}$ is almost surely one-to-one, then $P\left(\gap(\mathcal{T}_{\delta}^{(K)}) > 0 \right) = 1$ and thus there exists a $\d'>0$ such that $P\left(\gap(\mathcal{T}_{\delta}^{(K)}) > \frac{\d'}{2\s} \right) \geq 1-\frac{\eta}{3}$, and the conclusion of the lemma then follows easily. 
\end{proof}

\begin{remark}\label{rem:Tgap}
Note that for $\d>0$ fixed and $n$ sufficiently large, on the event $R_{K,n}$ we have that $\tau_{\fl{xn},\fl{2\s h\sqrt{n}} } \leq n^{3/2}$ only if $(x,h) \in [-K,K] \cap [0,K]$. 
Therefore, on $R_{K,n}$ we have that $\mathcal{T}_{\d,n} \cap [0,1] = \mathcal{T}_{\d,n}^{(K)} \cap [0,1]$. In particular, we have that 
\[
 R_{K,n} \cap \left\{ \gap(\mathcal{T}_{\d,n}^{(K)} ) > \d' \right\}
 \subset \left\{ \gap( \mathcal{T}_{\d,n} \cap [0,1] ) > \d' \right\}. 
\]
\end{remark}

For $n\geq 1$ recall that $\mathcal{X}_n$ is the ${\cal D}_+$-valued stochastic process defined by $\mathcal{X}_n(t) = \frac{X_{\fl{t n^{3/2}}}}{(2\s)^{-2/3}n}$.

\begin{proposition}\label{prop:Xtight}
 The sequence of $\mathcal{D}_+$-valued random variables $(\mathcal{X}_n)_{n\geq 1}$ is tight (with respect to the Skorokhod $J_1$ topology on ${\cal D}_+$), and moreover any subsequential limit is concentrated on paths that are continuous.  
\end{proposition}

\begin{proof}
It is enough to show that for any $\e,\eta>0$ there is a $\delta'>0$ such that 
\begin{equation}\label{tightcond}
 \limsup_{n\to\infty} P\Bigg( \max_{\substack{k,\ell\leq n^{3/2} \\ |k-\ell|\leq \delta' n^{3/2}}} |X_k - X_\ell| > \e n \Bigg) 
 < \eta. 
\end{equation}
To this end, we will later show that we can choose parameters $K,\d,$
and $\d'$ such that the conclusions to Lemmas \ref{lem:Krange} and
\ref{lem:gap} hold.  Thus, we will only need to control the
fluctuations of the path of the walk on the event
$R_{K,n} \cap \{ \gap(\mathcal{T}_{\d,n}^{(K)}) > \d' \}$.  For any
times $k<\ell\leq n^{3/2}$ with $|k-\ell|\leq \d' n^{3/2}$, it follows
from Remark \ref{rem:Tgap} that on the event
$R_{K,n} \cap \{ \gap(\mathcal{T}_{\d,n}^{(K)}) > \d' \}$ there are
times $\tau,\tau' \in \mathcal{T}_{\d,n}^{(K)}$ and
$\tau'' \in \mathcal{T}_{\d,n}$ that are ``consecutive'' in the set
$\mathcal{T}_{\d,n}$ in that
$\mathcal{T}_{\d,n} \cap \left((\tau,\tau')\cup(\tau',\tau'')\right) =
\emptyset$ and such that either (i)
$\tau n^{3/2}\leq k < \ell < \tau' n^{3/2}$ or (ii)
$ \tau n^{3/2} \leq k < \tau' n^{3/2}\le \ell < \tau''n^{3/2}$.
If one also has that $|X_k-X_\ell| > \e n$ then this implies that
either there is a $j \in (\tau n^{3/2},\tau' n^{3/2})$ with
$|X_j-X_{\tau n^{3/2}}| > \e n/4$ or there is a
$j \in (\tau'n^{3/2},\tau''n^{3/2})$ with
$|X_j-X_{\tau'n^{3/2}}|>\e n/4$. Indeed, if
  $|X_k-X_\ell|>\epsilon n$ and (i) is true then by the triangle
  inequality 
  $\max\{|X_\ell-X_{\tau n^{3/2}}|,|X_k-X_{\tau n^{3/2}}|\}>\epsilon
  n/2>\epsilon n/4$. If $|X_k-X_\ell|>\epsilon n$ and (ii) is true
  then
  \[\epsilon n<|X_k-X_\ell|\le |X_k-X_{\tau n^{3/2}}|+|X_{\tau
      n^{3/2}}-X_{\tau' n^{3/2}-1}|+1+|X_{\tau' n^{3/2}}-X_\ell|,\] so at least one of the terms on the right should be larger that $\epsilon n/4$.

Thus, for all sufficiently large $n$ we have that 
\begin{equation}\label{Hunion}
  \Bigg\{ \max_{\substack{k,\ell\leq n^{3/2} \\ |k-\ell|\leq \delta' n^{3/2}}} |X_k - X_\ell| > \e n \Bigg\} \cap R_{K,n} \cap \left\{ \gap(\mathcal{T}_{\d,n}^{(K)}) > \d' \right\} 
  \subset
  \bigcup_{\substack{x \in \d \Z \cap [-K,K]\\h \in \sqrt{\d}\Z \cap [0,K]}} A^{(K)}_{x,h,\d,\e,n}, 
\end{equation}
where 
\begin{align*}
 A^{(K)}_{x,h,\d,\e,n} 
 &= \left\{ \exists j> \tau^n_{x,h}  : \, \left|X_j - X_{\tau^n_{x,h}} \right| > \frac{\e}{4} n \text{ and } 
 \mathcal{T}_{\d,n} \cap \left( \frac{\tau^n_{x,h}}{n^{3/2}}, \frac{j}{n^{3/2}} \right) = \emptyset 
 \right\} \\
 &\makebox[5cm]{\ }\cap \left\{ \max_{i \in \Z}  L(\tau^n_{x,h}, i ) \leq 2\s K \sqrt{n} \right\}.
\end{align*}
We claim that there exists a constant $c_0>0$ which does not depend on $x,h,\d$ or $\e$ such that for all $K>0$, $x \in \d \Z \cap [-K,K]$, $h \in \sqrt{\d}\Z \cap [0,K]$, and $n$ sufficiently large we have  
\begin{equation}\label{PHbound}
 P(A^{(K)}_{x,h,\d,\e,n}) \leq 2 (1-c_0)^{\fl{\frac{\e}{5\d}}}. 
\end{equation}
Postponing the proof of \eqref{PHbound} for now, we will show how to use this to finish the proof of \eqref{tightcond}. 
Since there are at most $(2K/\delta+1)(K/\sqrt{\d}+1) \leq \frac{6K^2}{\d^{3/2}}$ indices $(x,h)$ in the union on the right side of \eqref{Hunion}, it follows from \eqref{Hunion} and \eqref{PHbound}
that  for $n$ sufficiently large we have
\begin{align*}
 P\Bigg( \max_{\substack{k,\ell\leq n^{3/2} \\ |k-\ell|\leq \delta' n^{3/2}}} |X_k - X_\ell| > \e n \Bigg)
\leq P(R_{K,n}^c) + P\left( \gap(\mathcal{T}_{\d,n}^{(K)} ) \leq \d' \right) 
 + \frac{12 K^2}{\d^{3/2}}(1-c_0)^{\fl{\frac{\e}{5\d}}}.
\end{align*}
Given $\e,\eta>0$ we can first choose $K$ large enough (depending only on $\eta$) so that $P(R_{K,n}^c) \leq \frac{\eta}{3}$. Next, we choose $\d>0$ small enough (depending on $\e$ and $K$) so that $\frac{12 K^2}{\d^{3/2}}(1-c_0)^{\fl{\frac{\e}{5\d}}} < \frac{\eta}{3}$, and finally we choose $\d'>0$ depending on $K, \eta$ and $\d$ as in Lemma \ref{lem:gap} so that $P( \gap(\mathcal{T}_{\d,n}^{(K)} \leq \d' ) \leq \frac{\eta}{3}$. 
This completes the proof of \eqref{tightcond}, pending the proof of \eqref{PHbound}. 

We now return to the proof of \eqref{PHbound}. 
The event $A^{(K)}_{x,h,\d,\e,n}$ involves a long excursion either to the right or left of $\fl{xn}$. To simplify things we will consider an event which only has long excursions to the right; the corresponding event with long excursions to the left is handled similarly. 
More precisely, letting  
\begin{align*}
 A^{(K),+}_{x,h,\d,\e,n} &= \left\{ \exists j> \tau_{x,h}^n: \, X_j  > \fl{xn} + \frac{\e}{4} n \text{ and } 
 \mathcal{T}_{\d,n} \cap \left(  \frac{\tau^n_{x,h}}{n^{3/2}}, \frac{j}{n^{3/2}} \right) = \emptyset 
 \right\} \\
 &\makebox[5.3cm]{\ } \cap \left\{ \max_{i \in \Z}  L( \tau_{x,h}^n, i ) \leq 2\s K \sqrt{n} \right\}, 
\end{align*}
we will show that there is a $c_0>0$ such that for all $K>0$,
$x \in \d \Z \cap [-K,K]$, $h \in \sqrt{\d}\Z \cap [0,K]$, and $n$ large
\begin{equation}\label{PH+bound}
 P(A^{(K),+}_{x,h,\d,\e,n}) \leq (1-c_0)^{\fl{\frac{\e}{5\d}}}. 
\end{equation}

To prove \eqref{PH+bound}, the key observation is that, on the
  event $A_{x,h,\d,\e,n}^{(K),+}$, for any
  $(x',h') \in (\d \Z)\times (\sqrt{\d}\Z_+)$ with
  $\tau_{x',h'}^n > \tau_{x,h}^n$ we must have that
  $\mu_{(x,h),(x',h')}^{n,+} > \frac{\fl{xn}}{n} + \frac{\e}{4}$
  (since $n \mu_{(x,h),(x',h')}^{n,+}$ is the maximum of the
    walk between times
  $\tau_{x,h}^n$ and $\tau_{x',h'}^n$) and thus
  $\Lambda_{x,h}^n(y) < \Lambda_{x',h'}^n(y)$ for all
  $\fl{xn}\leq \fl{yn} \leq \fl{xn} + \frac{\e n}{4}$.
In particular, for any $k \in \Z$ with $\fl{xn} \leq \fl{k\d n} < \fl{(k+1)\d n} \leq \fl{xn}+\frac{\e n}{4}$,  let $\ell_k$ be the non-negative integer such that $\tau_{k\d,\ell_k\sqrt{\d}}^n \leq \tau_{x,h}^n < \tau_{k\d, (\ell_k+1)\sqrt{\d}}^n$, or equivalently 
\[
 \Lambda_{k\d,\ell_k\sqrt{\d}}^n(k\d)
 \leq \Lambda_{x,h}^n(k \delta) < 
 \Lambda_{k\d,(\ell_k+1)\sqrt{\d}}^n(k\d).
 \]
(Note also that the definition of $A_{x,h,\d,\e,n}^{(K),+}$ implies that $\ell_k \sqrt{\d} \leq K$.)
Then on the event $A_{x,h,\d,\e,n}^{(K),+}$ we have that 
\[
 \mu_{(k\d,\ell_k\sqrt{\d}),(k\d,(\ell_k+1)\sqrt{\d})}^{n,+} 
 \geq \mu_{(x,h),(k\d,(\ell_k+1)\sqrt{\d})}^{n,+}
 \geq \frac{\fl{xn}}{n} + \frac{\e}{4} 
 \geq \frac{\fl{(k+1)\d n}}{n}. 
\]
(See Figure \ref{fig:snake} for a visual depiction of this.)

\begin{figure}[h]
  \includegraphics[scale=0.75]{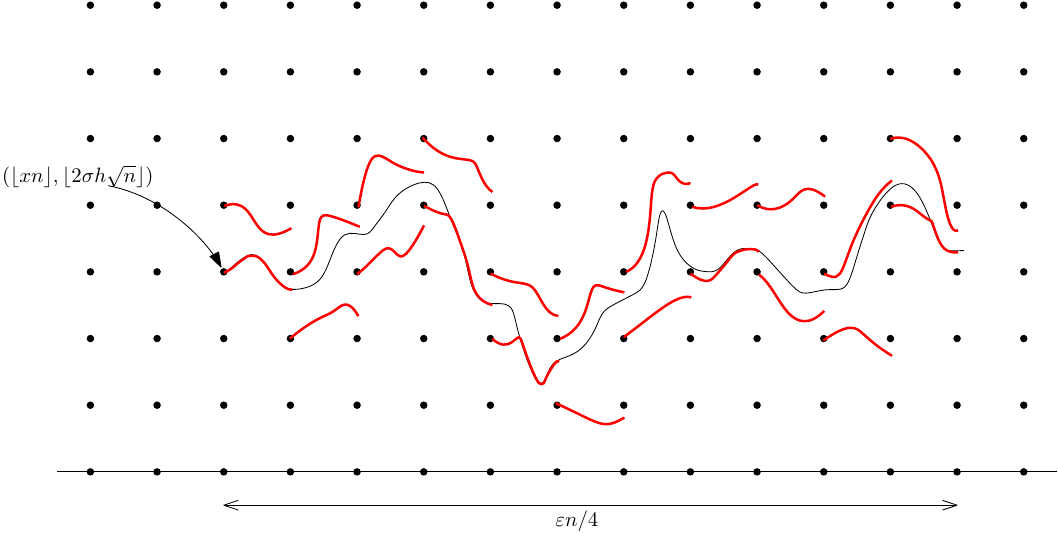}
  \caption{A visual depiction of the event $A^{(K),+}_{x,h,\d,\e,n}$. The black curve in the middle depicts $\Lambda_{x,h}^n(\cdot)$. The short red curves depict pieces of the curves $\Lambda_{k\d,\ell_k \sqrt{\d}}^n(\cdot)$ and $\Lambda_{k\d,(\ell_k+1)\sqrt{\d}}^n(\cdot)$ which on the event $A^{(K),+}_{x,h,\d,\e,n}$ do not meet.}\label{fig:snake}
\end{figure}

Therefore, if we define
\[
 S_{x,h,k} = \bigcup_{0\leq \ell \leq \frac{K}{\sqrt{\d}} } \left\{ \ell \sqrt{\d} \leq \Lambda_{x,h}(k \delta) < (\ell+1)\sqrt{\d}, \, 
  \mathfrak{m}_{(k\d,\ell\sqrt{\d}),(k\d,(\ell+1)\sqrt{\d})}^{+} \geq (k+1)\d \right\},
\]
then it follows from the Joint GRKT (Theorem \ref{thm:JRK-TSAW}) that 
\begin{equation}\label{PA+bound}
\limsup_{n\to\infty} P\left( A_{x,h,\d,\e,n}^{(K),+} \right)
\leq 
P\left( \bigcap_{k=\frac{x}{\d}}^{\frac{x}{\d} + \fl{\frac{\e}{5\d}}} S_{x,h,k} \right).
\end{equation}
Note that if we fix some $k_0 > x/\delta$ then the event $\bigcap_{k=\frac{x}{\d}}^{k_0-1}  S_{x,h,k}$ 
contains information about $\Lambda_{x',h'}(y)$ for finitely many pairs $(x',h') \in (\d\Z) \times (\sqrt{\d}\Z_+)$ and $y \in [x', k_0\delta]$. 
On the other hand, the only part of the event $S_{x,h,k_0}$ that depends on this information concerns the value of $\Lambda_{x,h}(k_0\d)$. Therefore, it follows from the Markov property that 
\begin{align*}
 P\left( S_{x,h,k_0} \, \biggl| \, \bigcap_{k=\frac{x}{\d}}^{k_0-1} S_{x,h,k} \right) 
 &= \sum_{0\leq \ell \leq \frac{K}{\sqrt{\d}} }
 P\left( \ell \sqrt{\d} \leq \Lambda_{x,h}(k_0\d) < (\ell+1)\sqrt{\d} \, \biggl| \, \bigcap_{k=\frac{x}{\d}}^{k_0-1} S_{x,h,k} \right) \\
 &\qquad\qquad\qquad\qquad \times P\left( \mathfrak{m}_{(k\d,\ell\sqrt{\d}),(k\d,(\ell+1)\sqrt{\d})}^{+} \geq (k+1)\d \right) \\ 
 &\leq \max_{\ell \leq \frac{K}{\sqrt{\d}} } P\left( \mathfrak{m}_{(k\d,\ell\sqrt{\d}),(k\d,(\ell+1)\sqrt{\d})}^{+} \geq (k+1)\d \right) \\
 &= \max_{\ell \leq \frac{K}{\sqrt{\d}} } P\left( \mathfrak{m}_{(k,\ell),(k,\ell+1)}^{+} \geq k+1 \right), 
\end{align*}
where the last equality follows from the scaling property of Brownian
motion.  
  Repeatedly applying this we get that
\begin{equation}\label{PcapSbound}
 P\left( \bigcap_{k=\frac{x}{\d}}^{\frac{x}{\d} + \fl{\frac{\e}{5\d}}} S_{x,h,k} \right)
 \leq \left\{ \sup_{k\in\Z} \sup_{\ell \in \Z_+} P\left( \mathfrak{m}_{(k,\ell),(k,\ell+1)}^{+} \geq k+1 \right) \right\}^{\left\lfloor \frac{\e}{5\d} \right\rfloor }
\end{equation}
By the properties of the construction of
the reflected/absorbed Brownian motions we have that
\[
 P\left( \mathfrak{m}_{(k,\ell),(k,\ell+1)}^{+} \geq k+1 \right)
 = \begin{cases}
    P\left( \mathfrak{m}_{(-1,\ell),(-1,\ell+1)}^{+} \geq 0 \right) & \text{if } k\leq -1 \\
    P\left( \mathfrak{m}_{(0,\ell),(0,\ell+1)}^{+} \geq 1 \right) & \text{if } k\geq 0. 
   \end{cases}
\]
All the probabilities on the right are less than 1 and as
$\ell \to \infty$ they converge to the probability that two
  independent Brownian motions started a distance 1 apart do not meet
  in 1 unit of time, which is also strictly less than 1.  Thus, it follows that there is a $c_0>0$ such that the supremum inside the braces in \eqref{PcapSbound} is less than $1-c_0$. Finally, it follows from \eqref{PA+bound} and \eqref{PcapSbound} that \eqref{PH+bound} holds for $n$ sufficiently large. 
\end{proof}

\begin{proposition}\label{prop:Htight}
 The sequence of $\mathcal{D}_+$-valued random variables $(\mathcal{H}_n)_{n\geq 1}$ is tight (with respect to the Skorokhod $J_1$ topology on ${\cal D}_+$), and moreover any subsequential limit is concentrated on paths that are continuous.  
\end{proposition}

\begin{proof}
  Recall \eqref{XH} and the notation $H_n=L(n,X_n)$. The idea of the
  proof is rather simple. Since Proposition \ref{prop:Xtight} implies
  that the walk doesn't move a large distance in a short time, in
  order for the local time process $(H_n)_{n\in\Z_+}$ to
 change by much in a short time either (1) the local time profile
  has to have a large fluctuation in a short distance or (2) the
  local time at a fixed location should increase very rapidly in
  a short amount of time. We will use the joint GRKT, Theorem
    \ref{thm:JRK-TSAW}, to show that both of these are unlikely to
  occur.

  It is enough to show that for
  any $\e, \eta>0$ there is a $\d'>0$ such that
\begin{equation}\label{tightcond-H}
 \limsup_{n\to\infty} P\Bigg( \max_{\substack{k,\ell\leq n^{3/2} \\ |k-\ell|\leq \delta' n^{3/2}}} |H_k - H_\ell| > \e \sqrt{n} \Bigg) 
 < \eta. 
\end{equation}
We note that similarly to \eqref{Hunion} 
\begin{equation}\label{Gunion}
  \Bigg\{ \max_{\substack{k,\ell\leq n^{3/2} \\ |k-\ell|\leq \delta' n^{3/2}}} |H_k - H_\ell| \geq \e \sqrt{n}  \Bigg\} 
  \cap R_{K,n} \cap \left\{ \gap(\mathcal{T}_{\d,n}^{(K)}) > \d' \right\}  \subset
  \bigcup_{\substack{x \in \d \Z \cap [-K,K]\\h \in \sqrt{\d}\Z \cap [0,K]}} G^{(K)}_{x,h,\d,\e,n}, 
\end{equation}
where 
\begin{align*}
 G^{(K)}_{x,h,\d,\e,n} &= \left\{ \exists j > \tau_{x,h}^n: \, 
 \left| H_j - H_{\tau_{x,h}^n} \right| > \frac{\e}{4} \sqrt{n} \text{ and }  \mathcal{T}_{\d,n} \cap \left( \frac{\tau_{x,h}^n}{n^{3/2}}, \frac{j}{n^{3/2}} \right) = \emptyset \right\}
  \\
 &\makebox[5cm]{\ }\cap \left\{ \max_{i \in \Z}  L(\tau^n_{x,h} , i ) \leq 2\s K \sqrt{n} \right\}.
\end{align*}
To bound $P(G^{(K)}_{x,h,\d,\e,n})$ we will fix an $M\geq 1$ (to be chosen later) and then use \eqref{PHbound} to conclude that for $n$ sufficiently large we have 
\begin{equation}\label{PG-PGH}
 P(G^{(K)}_{x,h,\d,\e,n}) 
 \leq 2(1-c_0)^{\fl{\frac{4\e^2}{5\d M}}} + P\left( G^{(K)}_{x,h,\d,\e,n} \setminus A^{(K)}_{x,h,\d,\frac{4\e^2}{M},n} \right). 
\end{equation}
To bound the last probability, note that on the event $G^{(K)}_{x,h,\d,\e,n}\setminus  A^{(K)}_{x,h,\d,\frac{4\e^2}{M},n}$ 
not only does the local time at the present location change by at least $\frac{\e}{4}\sqrt{n}$ before the next stopping time of the form $\tau_{x',h'}^n $ with $(x',h') \in \d \Z \times \sqrt{\d} \Z_+$, but it does so without moving a distance of more than $\frac{\e^2}{M} n$.  
Now, since 
\begin{align*}
 \left| H_j - H_{\tau_{x,h}^n} \right| 
 &= \left| L(j,X_j) - L\left(\tau_{x,h}^n, \fl{xn} \right) \right| \\
 &\leq  \left| L\left(j,X_j \right) - L(\tau_{x,h}^n,X_j) \right|
 + \left| L(\tau_{x,h}^n,X_j) - L\left(\tau_{x,h}^n, \fl{xn} \right) \right|, 
\end{align*}
we see that if the event $G^{(K)}_{x,h,\d,\e,n} \setminus A^{(K)}_{x,h,\d,\frac{4\e^2}{M},n}$ occurs, then either 
\begin{enumerate}[label=(\roman*)]
 \item there is an integer $k$ 
with $|k-\fl{xn}| \leq n(\e^2/M) $ such that $|L(\tau_{x,h}^n, k) - L(\tau_{x,h}^n, \fl{xn}) | \geq \frac{\e}{8}\sqrt{n}$, or equivalently, 
$\sup\limits_{y: |y-x| \leq \e^2/M} |\Lambda_{x,h}^n(y) - \Lambda_{x,h}^n(x) | \geq \frac{\e}{16\s}$,
 \item or there is a time $j > \tau_{x,h}^n$ such that  $|L(j,X_j) - L(\tau_{x,h}^n,X_j)| \geq \frac{\e}{8}\sqrt{n}$ and such that  there are no stopping times $\tau_{x',h'}^n \in (\tau_{x,h}^n, j)$ with $(x',h') \in \d \Z \times \sqrt{\d}\Z_+$.
\end{enumerate}
In the second case, 
let $k \in \Z$ be such that
if $X_j \geq 0$ then $\fl{k\d n} \leq X_j < \fl{(k+1)\d n}$ or if $X_j < 0$ then $\fl{(k-1)\d n} < X_j \leq \fl{k\d n}$, and then let $\ell \in \Z_+$ be such that 
$\tau_{k\d,\ell\sqrt{\d}}^n \leq \tau_{x,h}^n < j < \tau_{k\d,(\ell+1)\sqrt{\d}}^n$.
Moreover, note that the conditions in the event $G^{(K)}_{x,h,\d,\e,n} \setminus A^{(K)}_{x,h,\d,\frac{4\e^2}{M},n}$ ensure that $|k| \leq (K+\frac{\e^2}{M})/\d \leq 2K/\delta$ and $\ell \leq K/\sqrt{\d}$. 
Then, it follows that
\[
 L\left(\tau_{k\d,(\ell+1)\sqrt{\d}}^n, X_j \right) - L\left(\tau_{k\d,\ell\sqrt{\d}}^n, X_j \right) 
 \geq L(j,X_j) - L(\tau_{x,h}^n,X_j) \geq \frac{\e}{8}\sqrt{n}. 
\]
If $X_j\geq 0$ then this implies that 
$\sup_{y\in [k\d,(k+1)\d)} \Lambda_{k\d,(\ell+1)\sqrt{\d}}^n(y) - \Lambda_{k\d,\ell\sqrt{\d}}^n(y) \geq \frac{\e}{16\s}$ while if $X_j < 0$ then it implies that 
$\sup_{y\in ((k-1)\d,k\d]} \Lambda_{k\d,(\ell+1)\sqrt{\d}}^n(y) - \Lambda_{k\d,\ell\sqrt{\d}}^n(y) \geq \frac{\e}{16\s}$.
Therefore, it follows from the above consideration together with  Theorem \ref{thm:JRK-TSAW} that 
\begin{align}
 &\limsup_{n\to\infty} P\left( G^{(K)}_{x,h,\d,\e,n} \setminus A^{(K)}_{x,h,\d,\frac{4\e^2}{M},n} \right) \leq  P\left( \sup_{y: |y-x| \leq \frac{\e^2}{M}} |\Lambda_{x,h}(y) - h | \geq \frac{\e}{16\s} \right) \label{Bigchange} \\
 &\quad + \frac{12 K^2}{\d^{3/2}} \max_{\substack{0\leq k \d \leq K+\frac{\e^2}{M} \\ 0\leq \ell\sqrt{\d} \leq K }} P\left( \sup_{y\in [k\d, (k+1)\d)} \Lambda_{k\d, (\ell+1)\sqrt{\d}}(y) - \Lambda_{k\d, \ell\sqrt{\d}}(y) \geq \frac{\e}{16\s} \right).  \label{Fastchange}
\end{align}
Recalling that the curves $\Lambda_{x,h}(\cdot)$ are reflected/absorbed Brownian motions, it follows that the second probability in \eqref{Bigchange} can be bounded in terms of a Brownian motion $B(\cdot)$:
\begin{equation*}
 P\left( \sup_{y: |y-x| \leq \frac{\e^2}{M}} |\Lambda_{x,h}(y) - h | \geq \frac{\e}{16\s} \right)
 \leq 4P\left( \sup_{0\leq t\leq \frac{\e^2}{M}} B(t) \geq \frac{\e}{16\s} \right) 
 = 8 \overline\Phi\left( \frac{\sqrt{M}}{16\s} \right).
\end{equation*}
Similarly, for the probabilities in \eqref{Fastchange}, since $\Lambda_{k\d,\ell\sqrt{\d}}(\cdot)$ and $\Lambda_{k\d,(\ell+1)\sqrt{\d}}(\cdot)$ are independent coalescing, reflected/absorbed Brownian motions then it follows that (as long as $\sqrt{\d} < \frac{\e}{16\s}$) 
\begin{multline*}
 P\left( \sup_{y\in [k\d, (k+1)\d)} \Lambda_{k\d, (\ell+1)\sqrt{\d}}(y) - \Lambda_{k\d, \ell\sqrt{\d}}(y) \geq \frac{\e}{16\s} \right)
\\ \leq P\left( \sup_{0\leq t < \d} \left(\sqrt{\d} + \sqrt{2} B(t)\right) \geq \frac{\e}{16\s} \right) 
= 2 \overline\Phi\left(\frac{1}{\sqrt{2}} \left( \frac{\e}{16\s \sqrt{\d}} - 1 \right) \right). 
\end{multline*}
Combining these estimates with \eqref{Gunion}, \eqref{PG-PGH}, \eqref{Bigchange} and \eqref{Fastchange} we get that 
\begin{align}
  \limsup_{n\to\infty} &P\left( \max_{\substack{k,\ell\leq n^{3/2} \\ |k-\ell|\leq \delta' n^{3/2}}} |H_k - H_\ell| > \e \sqrt{n} \right) \leq \limsup_{n\to\infty} \left( P(R_{K,n}^c) + P\left( \gap(\mathcal{T}_{\d,n}^{(K)} ) \leq \d' \right) \right)  \nonumber\\
&\qquad + \frac{6 K^2}{\d^{3/2}} \left\{ 2(1-c_0)^{\fl{\frac{4\e^2}{5\d M}}} + 8 \overline\Phi\left( \frac{\sqrt{M}}{16\s} \right) +  \frac{24 K^2}{\d^{3/2}}  \overline\Phi\left( \frac{1}{\sqrt{2}} \left( \frac{\e}{16\s \sqrt{\d}} - 1 \right) \right) \right\} .  \label{KMd}
\end{align}
Finally, we show how the parameters $K$, $M$, $\d$, and $\d'$ can be carefully chosen to make the right side smaller than $\eta$. 
\begin{enumerate}[label=(\arabic*), nosep]
 \item First, choose $K = K(\eta)$ as in  Lemma \ref{lem:Krange}. 
 \item Next, choose $M\geq 1$ and $\d>0$ (depending on $K,\e,$ and $\eta$) so that \eqref{KMd} is less than $\eta/3$.  For instance, this can be done by letting $M=\d^{-1/2}$ and letting $\d>0$ be sufficiently small. 
 \item Finally, choose $\d'>0$ (depending on $K$ and $\d$) as in Lemma \ref{lem:gap}. 
\end{enumerate}
\end{proof}

\appendix

\section{Associated urn processes}\label{sec:urn}

In this section we describe generalized Polya urn processes which were
originally introduced in \cite{tTSAW} to study the TSAW. In
particular, we recall how these urn processes are connected with
the processes of directed edge local times that play an important role
in the proof of the joint GRKT for the TSAW (see
Section~\ref{sec:JRK}).
With the exception of Proposition~\ref{coup} for $N>1$, the content of
this section is a proper subset of \cite[Sections 2 and
3]{tTSAW}. We only set up the notation and review several basic facts
that are needed for the proof of Proposition~\ref{coup} with $N>1$,
which is given at the end of the section. We shall return to this
framework in Appendix~\ref{sec:overshoot}.

Consider a generalization of the Polya urn process where if there are
$r$ red balls and $b$ blue balls, then instead of drawing the balls
uniformly at random the next ball drawn is red (resp., blue) with
probability
$\frac{\l^{2r-1}}{\l^{2b} + \l^{2r-1}} = \frac{\l^{2(r-b)-1}}{1 +
  \l^{2(r-b)-1}}$ (resp. $\frac{1}{1+\l^{2(r-b)-1}}$) and then that
ball is replaced in the urn along with another ball of the same
color. Let $\mathfrak{R}_j$ and $\mathfrak{B}_j$ be the number of red
and blue balls, respectively, in the urn after $j$ draws, and let
$\mathfrak{D}_j = \mathfrak{R}_j-\mathfrak{B}_j$ be the discrepancy
between the number of red and blue balls. Since the probability of
drawing a red/blue ball depends only on the discrepancy, it is easy to
see that the discrepancy process $(\mathfrak{D}_j)_{j\geq 0}$ is a
Markov chain with transition probabilities
\[
 P(\mathfrak{D}_1 = i+1 \mid \mathfrak{D}_0 = i) = 1 - P(\mathfrak{D}_1 = i-1 \mid \mathfrak{D}_0 = i)
 = \frac{\l^{2i-1}}{1+\l^{2i-1}}, \qquad \forall i \in \Z. 
\]

This generalized urn process is connected with the TSAW as follows. If
we fix a site $k \in \Z\setminus \{0\}$, then the sequence of
right/left steps from $k$ on subsequent visits to $k$ has the same
distribution as the urn discrepancy process
$(\mathfrak{D}_j)_{j\geq 0}$ (where a draw of a red ball corresponds
to a step to the right and a blue ball corresponds to a step to the
left) if we use the initial condition $\mathfrak{D}_0 = 0$ for sites
$k\geq 1$ and the initial condition $\mathfrak{D}_0 = 1$ for sites
$k\leq -1$.  Similarly, the sequence of right/left steps from the
origin correspond to a slightly different urn process
$(\widetilde{\mathfrak{R}}_j, \widetilde{\mathfrak{B}}_j), j\geq
  0$, where the discrepancy
$\widetilde{\mathfrak{D}}_j = \widetilde{\mathfrak{R}}_j -
\widetilde{\mathfrak{B}}_j$ has transition probabilities given by
\[
 P(\widetilde{\mathfrak{D}}_1 = i+1 \mid \widetilde{\mathfrak{D}}_0 = i) = 1 - P(\widetilde{\mathfrak{D}}_1 = i-1 \mid \widetilde{\mathfrak{D}}_0 = i)
 = \frac{\l^{2i}}{1+\l^{2i}}, \qquad \forall i \in \Z, 
\]
and initial condition $\widetilde{\mathfrak{D}}_0 = 0$. 
Moreover, since transition probabilities for the TSAW only depend on the number of left and right steps from the current location, it follows that one can generate all of the steps of the TSAW at each site by running independent copies of the appropriate urn discrepancy process at each site.

Note that the distribution of the sequence of red/blue draws for the urn process $(\mathfrak{R}_j, \mathfrak{B}_j), j\geq 0$, depends only on the initial discrepancy $\mathfrak{D}_0$ and not on the particular values of $\mathfrak{R}_0$ and $\mathfrak{B}_0$.  Thus, below we will often use the notation $P_i(\cdot) = P(\cdot \mid \mathfrak{D}_0 = i)$ for probabilities involving the urn process $(\mathfrak{R}_j, \mathfrak{B}_j), j\geq 0$, started with discrepancy $\mathfrak{D}_0 = i$.

\subsection{Connection with directed edge local times}
In this subsection we will explain how the distribution of the Markov
chain of directed edge local times
$(\mathcal{E}_{\tau_{k,m}}^+(\ell))_{\ell\geq k}$ which appears in the
proof of the joint GRKT is related to the urn processes described
above.

Let 
\[
 \beta_i = \inf\{ j\geq 1: \, \mathfrak{B}_j-\mathfrak{B}_0 = i \},\quad \t{\beta}_i = \inf\{ j\geq 1: \, \widetilde{\mathfrak{B}}_j-\widetilde{\mathfrak{B}}_0 = i \}
\]
be the number of draws for the urn process $(\mathfrak{R}_j, \mathfrak{B}_j),j\geq 0$, (resp.\ $(\widetilde{\mathfrak{R}}_j, \widetilde{\mathfrak{B}}_j),j\geq 0$) until the $i$-th blue ball is drawn.

\begin{lemma}\label{lem:blp-transitions}
 For any $k\in\Z$ and $m\geq 0$, $\{\mathcal{E}^+_{\tau_{k,m}} (\ell) \}_{\ell\geq k}$ is a (non-homogeneous) Markov chain with initial distribution supported on $\{0,1,\ldots,m\}$ so that
 \begin{equation}\label{blp-ic}
  \forall j\in\{0,1,\ldots,m\}\qquad P( \mathcal{E}^+_{\tau_{k,m}}(k) = j ) 
  = \begin{cases}
     P_1\left( \mathfrak{D}_m = 2j-m+1 \right) & \text{if } k\leq -1,\\
     P_0\left( \widetilde{\mathfrak{D}}_m = 2j-m \right) & \text{if } k = 0,\\
     P_0\left( \mathfrak{D}_m = 2j-m \right) & \text{if } k\geq 1,
    \end{cases}
 \end{equation}
 and with transition probabilities given by 
 \begin{equation}\label{blp-tp}
 P\left( \mathcal{E}_{\tau_{k,m}}^+(\ell) = i+x \mid \mathcal{E}_{\tau_{k,m}}^+(\ell-1) = i \right) 
 = \begin{cases}
    P_1\left( \mathfrak{D}_{\beta_{i+1}} = x  \right) & \text{if } \ell \leq -1, \\
    P_0\left( \widetilde{\mathfrak{D}}_{\tilde\beta_{i+1}} = x-1  \right) & \text{if } \ell = 0,\\
    P_0\left( \mathfrak{D}_{\beta_{i}} = x \right) & \text{if } \ell \geq 1. 
   \end{cases}
\end{equation}
\end{lemma}
\begin{proof}
 For the distribution of the initial condition in \eqref{blp-ic}, note that at time $\tau_{k,m}$ the walk has taken $m$ previous steps away from site $k$. Thus, $\mathcal{E}^+_{\tau_{k,m}}(k)=j$ if and only if $j$ of these steps have been to the right and $m-j$ of these steps have been to the left. Translating this to the corresponding urn process this means that the change in discrepancy must be $\mathfrak{D}_m-\mathfrak{D}_0=2j-m$ (or $\widetilde{\mathfrak{D}}_m-\widetilde{\mathfrak{D}}_0=2j-m$ in the case $k=0$). 
 
 For the transition probabilities in \eqref{blp-tp}, we will explain only the case when $\ell \leq -1$ as the other cases are similar. 
 In this case since the walk 
starts at $X_0=0$ and ends at $X_{\tau_{k,m}} = k < \ell < 0$ then the walk will have taken one more step from $\ell$ to $\ell-1$ than it took from $\ell-1$ to $\ell$ by time $\tau_{k,m}$,
and so the information $\mathcal{E}_{\tau_{k,m}}^+(\ell-1)=i$ means that the distribution of $\mathcal{E}_{\tau_{k,m}}^+(\ell)$ is determined by the urn process $(\mathfrak{R}_j, \mathfrak{B}_j), j\geq 0$, 
for $j\leq \beta_{i+1}$. 
More precisely,  $\mathcal{E}_{\tau_{k,m}}^+(\ell) = i+x$ if and only if there are $i+x$ red balls drawn before the $(i+1)$-st blue ball drawn, or equivalently if 
$\mathfrak{D}_{\beta_{i+1}}-\mathfrak{D}_0 = x-1$.  Since at sites $\ell \leq -1$ the discrepancy urn process starts at $\mathfrak{D}_0 = 1$ then we have that 
$P\left( \mathcal{E}_{\tau_{k,m}}^+(\ell) = i+x \mid \mathcal{E}_{\tau_{k,m}}^+(\ell-1) = i \right)  = P_1\left( \mathfrak{D}_{\beta_{i+1}} = x  \right)$. 
\end{proof}

For the proof of the joint GRKT we will also need to know the joint distribution of multiple of these Markov chains corresponding to stopping times $\tau_{k_1,m_1},\tau_{k_2,m_2},\ldots,\tau_{k_N,m_N}$. 
We omit the proof of the lemma below since it is essentially the same as the proof of Lemma \ref{lem:blp-transitions}. 
\begin{lemma}\label{lem:blpmulti-tp}
 For any $N\geq 1$ and any $k_1,k_2,\ldots,k_N \in \Z$ and $m_1,m_2,\ldots,m_N \geq 0$, the joint process $\left(
 \mathcal{E}^+_{\tau_{k_1,m_1}}(\ell),\mathcal{E}^+_{\tau_{k_2,m_2}}(\ell),\ldots,\mathcal{E}^+_{\tau_{k_N,m_N}}(\ell)
 \right), \ell\geq \max\{k_1,\ldots,k_N\}$, is a (non-homogeneous) Markov chain with transition probabilities given by  
\begin{align*}
 &P\left( \mathcal{E}^+_{\tau_{k_r,m_r}}(\ell) =i_r+x_r, \, r=1,2,\ldots N \mid \mathcal{E}^+_{\tau_{k_r,m_r}}(\ell-1) =i_r ,\, r=1,2,\ldots N\right) 
 \\
 &\qquad = 
 \begin{cases}
 P_1\left( \mathfrak{D}_{\beta_{i_r+1}} = x_r, \, r=1,2,\ldots N \right)  &\text{if } \ell \leq -1, \\
 P_0\left( \widetilde{\mathfrak{D}}_{\tilde\beta_{i_r+1}} = x_r-1, \, r=1,2,\ldots N \right) &\text{if } \ell = 0, \\
 P_0\left( \mathfrak{D}_{\beta_{i_r}} = x_r, \, r=1,2,\ldots N \right) &\text{if } \ell \geq 1.
 \end{cases}
\end{align*}
\end{lemma}

\subsection{Stationary distribution and convergence}
The previous subsection showed the connection of the directed edge
local times of the TSAW with the discrepancies in urn processes. Since
the joint GRKTs involve scaling limits of the directed edge local
times, it will be important to understand the distribution of
$\mathfrak{D}_j$ and $\mathfrak{D}_{\beta_j}$ for $j$ large.  The
asymptotics of distributions of $\widetilde{\mathfrak{D}}_j$ and
$\widetilde{\mathfrak{D}}_{\tilde\beta_j}$ are less important since
they are only used for a single step of the directed edge local time
process, but we will record these as well.

Since $(\mathfrak{D}_j)_{j\ge 0}$ is a birth-death Markov chain, it is easy to use the detailed balance equations to show that it has stationary distribution $\rho(x) = C' \l^{x(x-2)}\left( 1+ \l^{2x-1} \right)$, $x \in \Z$, where $C'>0$ is an appropriate normalizing constant. Similarly, 
the stationary distribution for $(\widetilde{\mathfrak{D}}_j)_{j\ge 0}$ is $\tilde\rho(x) = \t{C}' \l^{x(x-1)}(1+\l^{2x})$, $x \in \Z$. 

Since the discrepancy process $(\mathfrak{D}_j)_{j\ge 0}$ is positive recurrent, the urn will asymptotically have roughly the same number of red and blue balls. In fact, it is not hard to show from this that $\lim_{j\to\infty} \frac{\beta_j}{j} = 2$, and from this we can use the ergodic theorem for Markov chains to identify the stationary distribution for $(\mathfrak{D}_{\beta_j})_{j\ge 0}$ as
\begin{align*}
 \pi(x) = \lim_{j\to\infty} \frac{1}{j} \sum_{i=1}^j \ind{\mathfrak{D}_{\beta_i} = x} 
 &= \lim_{j\to\infty} \left( \frac{\beta_j}{j} \right)\left( \frac{1}{\beta_j}  \sum_{m=1}^{\beta_j} \ind{\mathfrak{D}_{m-1}=x+1, \, \mathfrak{D}_m = x} \right) \\
 &= 2 \rho(x+1) \frac{1}{1+\l^{2x+1}} 
 = 2 C'\l^{x^2-1}=:C\lambda^{x^2}, \qquad \forall x \in \Z. 
\end{align*}
Similarly, the stationary distribution for $(\widetilde{\mathfrak{D}}_{\tilde\beta_j})_{j\ge 0}$ is 
$\tilde{\pi}(x) = 2 \t{C}' \l^{(x+1)x}$.
Moreover, the convergence of the distribution of $\mathfrak{D}_{\beta_j}$ to $\pi$ is exponentially fast. In fact, the following proposition gives exponential bounds on the total variation distance of the distribution of $(\mathfrak{D}_{\beta_{j_1}}, \mathfrak{D}_{\beta_{j_2}}, \ldots, \mathfrak{D}_{\beta_{j_N}} )$ from $\pi^{\otimes N}$. 
We have already stated this fact in Proposition~\ref{coup} using the
equivalent notation (see Lemma~\ref{lem:blpmulti-tp}) of directed edge
local times. A similar result can also be proven for the Markov chain
$(\widetilde{\mathfrak{D}}_{\tilde\beta_{j}})_{j\ge 0}$, but we will
not need such a result in the present paper.

\begin{namedthm*}{Proposition \ref{coup} (restated)}
 There are constants $C_2$ and $C_3$ such that for any $N\in\N$ and any choice of $0=n_0<n_1 < n_2 < n_3 < \cdots < n_N$ we have for $x \in \{0,1\}$ that 
\begin{equation}\label{D-multiTV}
 \sum_{k_1,k_2,\ldots,k_N \in \Z} 
 \left| 
 P_x\left( \mathfrak{D}_{\beta_{n_i}} = k_i, \, i=1,2,\ldots,N \right) 
 - \prod_{i=1}^N \pi(k_i) \right| 
 \leq C_3 \sum_{i=1}^N e^{-C_2(n_i-n_{i-1})}. 
\end{equation}
\end{namedthm*}

\begin{proof}
  The one dimensional case $N=1$ was proved in \cite[Lemma 1]{tTSAW}.
  In fact, in the proof of that Lemma it was shown that the renewal
  times of the Markov chain $(\mathfrak{D}_{\beta_n})_{n\ge 0}$ have
  exponential moments (see (3.13) in \cite{tTSAW}).  Thus, it follows
  from \cite[Theorem 6.14 and Example 5.5(a)]{Num84} that for every
  $j \in \Z$ there is a constant $M(j)<\infty$ such that
 \begin{equation}\label{TV-M}
   \sum_{k \in \Z} \left| P_j\left( \mathfrak{D}_{\beta_n} = k \right)  - \pi(k) \right| \leq M(j) e^{-C_2 n}, \qquad \forall n\geq 1,  
 \end{equation}
and, moreover, that the constants $M(j)$ are such that 
\begin{equation}\label{Mbar}
 \overline{M} := \sum_{k \in \Z} \pi(k) M(k) < \infty. 
\end{equation} 
The proof of \eqref{D-multiTV} with $C_3 = \max\{M(0), M(1), \overline{M} \}$ then follows rather easily from \eqref{TV-M} and \eqref{Mbar} using induction on $N\geq 1$. 
\end{proof}

\begin{remark}\label{rem:Dtails}
  We close this section by noting that the stationary distributions
  $\pi$ and $\tilde{\pi}$ can be used to give tail estimates on
  $\mathfrak{D}_{\beta_n}$ and
  $\widetilde{\mathfrak{D}}_{\tilde\beta_n}$ that are uniform in $n$.
  Indeed, we have
  $P_i( \mathfrak{D}_{\beta_n} = j ) \leq \frac{\pi(j)}{\pi(i)}$ and
  $P_i( \widetilde{\mathfrak{D}}_{\tilde\beta_n} = j ) \leq
  \frac{\tilde{\pi}(j)}{\tilde\pi(i)}$, for all $i,j \in \Z$ and
  $n\geq 1$.  Together with Lemma \ref{lem:blp-transitions} this
  implies that the increments of the directed edge local time
  processes $(\mathcal{E}_{k,m}^+(\ell))_{\ell\geq k}$ have uniform
  Gaussian tails.  This fact will be used in Appendix \ref{sec:aux} below.
\end{remark}
 
\section{Relating the backward and the forward paths}

The following is the analog of equation \eqref{backpath} for the TSRM. Note that this formula differs from the definition of the backward paths in \cite[p.\,389]{TW98} because this paper uses a slightly different version of stopping times. 
\begin{lemma}\label{lem:TSRMweave}
 For $y\neq x$ and $h\geq 0$, $\Lambda_{x,h}(y) = \inf\left\{ h'\geq 0: \Lambda_{y,h'}(x) > h \right\}$, almost surely.
\end{lemma}
\begin{proof}
  It follows from the definition of $\mathfrak{t}_{x,h}$ and the relation $h = \Lambda_{x,h}(x) = \mathfrak{L}(\mathfrak{t}_{x,h},x)$ that
\begin{equation}\label{txhprop}
 s > \mathfrak{t}_{x,h} \iff \mathfrak{L}(s,x) > h. 
\end{equation}
Next, note that 
  \begin{align*}
  \Lambda_{x,h}(y) = 0
  &\iff \mathfrak{L}(\mathfrak{t}_{x,h},y) = 0 = \mathfrak{L}(\mathfrak{t}_{y,0},y) \\
  &\iff \mathfrak{t}_{x,h} < \mathfrak{t}_{y,0} \qquad \text{(recalling that $\mathfrak{t}_{x,h} \neq \mathfrak{t}_{y,0}$, a.s.)}\\
  &\iff \mathfrak{t}_{x,h} < \mathfrak{t}_{y,h'} \quad \forall h'>0 \qquad\text{(using $\mathfrak{t}_{y,0} = \lim_{h'\to 0^+} \mathfrak{t}_{y,h'}$ and  $\mathfrak{t}_{x,h} \neq \mathfrak{t}_{y,0}$ a.s.)} \\
  &\overset{\eqref{txhprop}}{\iff} \Lambda_{y,h'}(x) = \mathfrak{L}(\mathfrak{t}_{y,h'},x) > h, \quad \forall h'>0\\
  &\iff \inf\left\{ h'\geq 0: \Lambda_{y,h'}(x) > h \right\} = 0.
 \end{align*}
 Similarly, 
 if $\Lambda_{x,h}(y) = \mathfrak{L}(\mathfrak{t}_{x,h},y) = b$ then the definition of $\mathfrak{t}_{y,b}$ implies that $\mathfrak{t}_{x,h} \leq \mathfrak{t}_{y,b}$, but since these times are almost surely not equal then we have $\mathfrak{t}_{x,h} < \mathfrak{t}_{y,b}$. Moreover, for any $h' \in [0,b)$ we have that 
 $b > h' = \mathfrak{L}(\mathfrak{t}_{y,h'},y)$ so that $\mathfrak{t}_{y,h'} < \mathfrak{t}_{x,h}$ for all $h'\in [0,b)$. 
 Conversely, suppose that $\mathfrak{t}_{y,h'} < \mathfrak{t}_{x,h} < \mathfrak{t}_{y,b}$ for all $h' \in [0,b)$. The second inequality implies that $\Lambda_{x,h}(y) = \mathfrak{L}(\mathfrak{t}_{x,h},y) \leq  \mathfrak{L}(\mathfrak{t}_{y,b},y) = b$, while the first inequality implies (by \eqref{txhprop}) that $\Lambda_{x,h}(y)  = \mathfrak{L}(\mathfrak{t}_{x,h},y) > h'$. Since this holds for all $h' \in [0,b)$ these together imply that $\Lambda_{x,h}(y) = b$. 
 Using the above argument to get the first $\iff$ in the display below, we see that
 \begin{align*}
  \Lambda_{x,h}(y) = b
  &\iff \mathfrak{t}_{y,h'} < \mathfrak{t}_{x,h} < \mathfrak{t}_{y,b}, \quad \forall h' \in [0,b) \\
  &\iff \mathfrak{t}_{y,h'} < \mathfrak{t}_{x,h} < \mathfrak{t}_{y,h''}, \quad \forall h' \in [0,b) \text{ and } \forall h'' > b \\
  &\overset{\eqref{txhprop}}{\iff} \Lambda_{y,h'}(x) \leq h < \Lambda_{y,h''}(x), \quad \forall h' \in [0,b)  \text{ and } \forall h'' > b \\
  &\iff \inf\left\{ h'\geq 0: \Lambda_{y,h'}(x) > h \right\} = b. 
 \end{align*}
 Note that the second line holds since
 $h''\mapsto \mathfrak{t}_{y,h''}$ is right continuous by the
 definition in \eqref{txh} and the fact that
 $t\mapsto \mathfrak{L}(t,y)$ is almost surely continuous as shown in
 \cite{TW98}.  This completes the proof of the lemma.
\end{proof}

We record the following useful corollary. 
\begin{corollary}\label{cor:sandwich}
 For any $x\neq y$ and $h,h'\geq 0$, with probability 1 we have
\begin{equation}\label{TSRMsandwich}
\left\{ \Lambda_{y,h'}(x) < h \right\}
\subset
\left\{ \Lambda_{x,h}(y) > h' \right\}
\subset
\left\{ \Lambda_{y,h'}(x) \leq h \right\}
\subset
\left\{ \Lambda_{x,h}(y) \geq h' \right\}.
\end{equation}
\end{corollary}
\begin{proof}
First of all, since $\Lambda_{y,h'}(x) = \lim_{\e \to 0^+} \Lambda_{y,h'+\e}(x)$, almost surely (see Remark 2 on p.\,387 and Theorem 4.3 in \cite{TW98}), it follows that 
\begin{equation*}
 \Lambda_{y,h'}(x) < h \ 
 \Longrightarrow \ h' \notin \overline{\{ b\geq 0: \Lambda_{y,b}(x) > h \}}\  \overset{\text{Lem.\,} \ref{lem:TSRMweave}}{\iff} \ h' < \Lambda_{x,h}(y)
\end{equation*}
This proves the first inclusion in \eqref{TSRMsandwich}. 
The second inclusion follows easily from Lemma \ref{lem:TSRMweave} because $\Lambda_{x,h}(y) > h'$ implies that $h' \notin \{b\geq 0: \Lambda_{y,b}(x) > h \}$. 
Finally, the third inclusion in \eqref{TSRMsandwich} follows from the first inclusion by taking complements of both sets and reversing the roles of $(x,h)$ and $(y,h')$. 
\end{proof}

\section{Gambler's ruin estimates}\label{sec:aux}

The main result of this section is Lemma~\ref{gr3}. It gives gambler's
ruin sort of estimates for the difference of two discrete local time
curves. These estimates are not optimal but they are sufficient to
give us Corollary~\ref{close} which is a crucial technical result
needed for the proof of convergence of the merge points of a pair of
local time curves in the joint forward GRKT, Theorem
\ref{thm:JRK-forward}.

The preliminaries and the proof of Lemma~\ref{gr3} follow the
blueprint given in \cite[p.\,1545--1548]{tTSAW} for a single
curve. Since we need a result for the difference of two curves, which
is non-Markovian, we have to work with a two-dimensional process. This
creates additional technical difficulties. In particular, while the
overshoot estimates (A4.5), (A4.6) in \cite{tTSAW} follow immediately
from the exponential bounds on the right tails, the proof of an
analogous statement for the difference process, Lemma~\ref{os},
requires a non-trivial effort and is deferred to
Appendix~\ref{sec:overshoot}.

\subsection{Preliminaries}
Recall the notation for $S^n_i$ introduced in \eqref{S}. Parameters $n$ and $x<0$ will be fixed throughout this section, so we shall drop them from the notation. The cases $x=0$ and $x>0$ will not be considered separately as their treatment requires only minor changes.

We shall first consider a single process $S(j),\ j\in\Z_+$, associated
to the directed edge local times as defined in \eqref{S} without
specifying the initial state (and dropping all indices). Let
$\xi(j):=S(j)-S(j-1)$ be its increments, $\zeta(j),j\in\N$, be i.i.d.\
$\pi$-distributed random variables, and
$Y(j)=Y(0)+\sum_{i=1}^j\zeta(j)$, $j\in\N$. Set $S_*(j)=\min_{i\le j}S(i)$ and
$Y_*(j)=\min_{i\le j}Y(i)$.

The following elementary lemma is not new. We record it for reference
purposes.

\begin{lemma}[Coupling lemma]\label{notdecoup}  Let $b\ge 2$, $S(0)=Y(0)=b$, and the
  walks $S$ and $Y$ be coupled using the maximal coupling of their
  increments. Denote the decoupling time by
  $\gamma:=\min\{j\in\N:\xi(j)\ne \zeta(j)\}$. Then for any $\delta>0$
  \begin{align}
    &P(\gamma>b^{2-\delta})\to 1\ \ \text{as }b\to\infty.\label{decoup}\\
    &P(S_*(\fl{b^{2-\delta}})>b/2)\to 1\ \ \text{as }b\to\infty.\label{min}
  \end{align}
\end{lemma}

\begin{proof} Note that \eqref{min} follows from \eqref{decoup} and the invariance principle for random walks. We now prove \eqref{decoup}.
  \begin{align*}
    P(\gamma>b^{2-\delta})&\ge P\left(\gamma>b^{2-\delta},S_*(\fl{b^{2-\delta}})>b/2\right)=P\left(\gamma>b^{2-\delta},Y_*(\fl{b^{2-\delta}})>b/2\right)\\ &=P\left(\gamma>b^{2-\delta}\mid Y_*(\fl{b^{2-\delta}})>b/2\right)P\left(Y_*(\fl{b^{2-\delta}})>b/2\right)\\ &\ge \left(1-C_3e^{-C_2b/2}\right)^{\fl{b^{2-\delta}}}P\left(Y_*(\fl{b^{2-\delta}})-Y(0)>-b/2\right)\to 1\ \ \text{as}\ b\to\infty.
  \end{align*}
  In the last line we used Proposition~\ref{coup} with $N=1$ (proved already in
\cite[(3.3)-(3.5)]{tTSAW}) and the invariance
  principle for a mean $0$ random walk.
\end{proof}

Suppose now that we have two coalescing reflected/absorbed random
walks $S_1\le S_2$ as above constructed from directed edge local times
to the right from $x<0$ of the same TSAW.  Lemma~\ref{notdecoup}
  allows us to couple each $S_i$ separately to a random walk $Y_i$
  with independent $\pi$-distributed increments, $i=1,2$. We could try
  to couple a pair $(S_1,S_2)$ with $(Y_1,Y_2)$ where $Y_1$ and $Y_2$
  are independent. But this coupling will break down as $S_1$ and
  $S_2$ get close. Below we carefully analyze $S_2-S_1$ to get an
  upper bound on the coalescence time of $S_1$ and $S_2$ that we need
  for the proof of Theorem~\ref{thm:JRK-forward}.

Note that if $Y_1$ and $Y_2$ are independent, then $Y_2-Y_1$ is
a mean 0 random walk with independent increments distributed according
to $\pi_2(\cdot):=\sum_{y\in \Z}\pi(y)\pi(y+\cdot)$. The
  distribution $\pi_2$ is symmetric with a superexponential tail
decay, so the hitting time of $0$ by $Y_2-Y_1$ obeys standard
estimates (see, for instance, \cite[Section 5.1.1]{LL}). In our case,
while the joint process $(S_1(j), S_2(j))$, $j\geq 0$, is a Markov process, the difference of the walks
$S_\Delta(j):=S_2(j)-S_1(j),\ j\ge 0$, is not Markov, making the hitting of $0$ by $S_\Delta(\cdot)$ harder to control. Let
$\Delta(j):=S_2(j)-S_2(j-1)-(S_1(j)-S_1(j-1))$, $j\ge 1$, be the
increments of the difference of the two walks.

It is convenient to modify $S_1$ and $S_2$ at the absorption or
  coalescence times to prevent trapping of the modified walks and
  their difference in a bounded domain.\footnote{This is akin to
    non-trapping conditions \cite[(4.10), (4.11)]{tTSAW}.} We do it as
  follows.  If the walks coalesce before $S_1$ gets absorbed then
from the time they coalesce we let them follow two independent random
walks with the step distribution $\pi$; if $S_1$ gets absorbed before
the walks coalesce then we re-start $S_1$ from
1
. These modifications will not matter for our applications
. The modified processes will be
denoted by $\t{S}_1,\t{S}_2, \t{S}_\Delta:=\t{S}_2-\t{S}_1$. We also set
\[P_{n_1,m}(\cdot):= P(\cdot\mid \t{S}_1(0) = n_1,\t{S}_2(0)=n_1+m)\]
and similarly for the expectation.

Define
\begin{align*}
  p_\Delta(k|n_1,m):&=P_{n_1,m}(\Delta(1)=k)\\ &=\sum_{y\in\Z}P_{n_1,m}(\t{S}_1(1)=n_1+y,\t{S}_2(1)=n_1+m+y+k)
\end{align*}
Recall that $\pi_2(k)=P(\zeta(2)-\zeta(1)=k)$,
$\pi_2(k)=\pi_2(-k)$, and its tails are bounded by
$C\lambda^{k^2/2}$. From Proposition~\ref{coup} we get
\begin{equation}
  \sum_{k\in\Z}|p_\Delta(k|n_1,m)-\pi_2(k)|\le 2C_3e^{-C_2(n_1\wedge m)},\label{4.5}
\end{equation}
Since
$|E_{n_1,m}[\Delta(1)]|\le
|E_{n_1,m}[\t{S}_2(1)-(n_1+m)]|+|E_{n_1,m}[\t{S}_1(1)-n_1]|$,
Proposition~\ref{coup} (for $N=1$), the symmetry of $\pi$, and the
exponential decay of the tails of $\xi$ (see Remark \ref{rem:Dtails})
imply that
\begin{align}
 \Big|\sum_{k\in\Z}kp_\Delta(k|n_1,m)\Big|&= |E_{n_1,m}[\Delta(1)]| 
 \le C_4e^{-C_5 n_1}\ \ \text{with } C_5=C_2/2\ \ \text{and}\label{A4.1}\\ \liminf_{n_1\wedge m\to\infty}\sum_{k\in\Z}p_\Delta(k|n_1,m)e^{-C_5k}&\ge \sum_{k\in\Z}\pi_2(k)e^{-C_5k}\nonumber\\ =\pi_2(0)&+\sum_{k\in\N}\pi_2(k)\left(e^{-C_5k}+e^{C_5 k}\right)>\pi_2(0)+\sum_{k\ne 0}\pi_2(k)=1.   \label{A4.2}     
\end{align}
In particular, we can choose $a_1\in\N$ and a constant $C_6$ so that for $n_1\wedge m\ge a_1$,
\begin{equation}\label{A4.3}
  C_6\left(\sum_{k\in\Z}p_\Delta(k|n_1,m)e^{-C_5k}-1\right)\ge C_4.
\end{equation}

\begin{lemma} \label{subm}
  Let $a_1\in\N$ be chosen so that \eqref{A4.3} holds. Then for
  $n_1\ge m\ge a_1$
  \[E_{n_1,m}\left[\pm \t{S}_\Delta(1)+C_6e^{-C_5 \t{S}_\Delta(1)}\right]\ge \pm m+C_6 e^{-C_5 m}.\] Hence, as long as $\t{S}_1(j)\ge \t{S}_\Delta(j)\ge a_1$, the processes $\pm \t{S}_\Delta(j)+C_6e^{-C_5 \t{S}_\Delta(j)},\ j\ge 0$, are submartingales relative to the natural filtration of $(\t{S}_1,\t{S}_\Delta)$.
\end{lemma}

\begin{proof}
  By \eqref{A4.1} and \eqref{A4.3},
  \begin{multline*}
    E_{n_1,m}\left[\pm \Delta(1)+C_6e^{-C_5 m}(e^{-C_5 \Delta(1)}-1)\right]\ge \\-C_4e^{-C_5(n_1\wedge m)} +C_6e^{-C_5 m}\left(\sum_{k\in\Z}p_\Delta(k|n_1,n_1+m)e^{-C_5 k}-1\right)\ge 0.
  \end{multline*}
\end{proof}

\subsection{Gambler's ruin for a modified process} Fix an arbitrary integer $b\ge 2$. All quantities related to modified processes introduced below depend on $b$. But we shall not reflect this fact in our notation. Let $\alpha:=\inf\{j\ge 0: \t{S}_1(j)< b\}$. To adapt the proof of \cite[Lemma 3]{tTSAW} to our setting we shall first prove the result for the process $\t{S}_\Delta$ modified starting at time $\alpha$. 
Namely, 
\begin{align*}
    \c{S}_\Delta(\ell+1)-\c{S}_\Delta(\ell)&:=
  \begin{cases}
    \t{S}_2(\ell+1)-\t{S}_2(\ell)-(\t{S}_1(\ell+1)-\t{S}_1(\ell)),&\text{if }\ell<\alpha;\\ \chi(\ell+1),&\text{if }\ell\ge\alpha,
  \end{cases}
\end{align*}
where $\chi(\ell),\ell\in\N_0$, are i.i.d.\ uniform on $\{-1,1\}$
random variables independent of everything else. In words, once
$\t{S}_1$ falls below $b$, the process $\c{S}_\Delta$ ``forgets''
about $\t{S}_1$ and starts following an independent simple symmetric
random walk.

For an interval ${\cal I}\subset [0,\infty)$ define
$\c{\theta}_{\cal I}=\inf\{\ell\ge 0: \c{S}_\Delta(\ell)\not\in{\cal I}\}$. The bounds stated in the next lemma are the analogs of
(A4.5) and (A4.6) in \cite{tTSAW}. 
\begin{lemma}[Overshoot lemma]\label{os}
  There is a constant $C_7>0$ such that for all $b\ge 2$, $a\in[0,b-2]\cap \Z$, 
$n_1\in\Z_+$, and $m\in (a,b)\cap\Z$
\begin{align}
  E_{n_1,m}(\c{S}_\Delta(\c{\theta}_{(a,b)})| \c{S}_\Delta(\c{\theta}_{(a,b)})\ge b)&\le b+C_7;\label{A4.5}\\
  E_{n_1,m}(\c{S}^2_\Delta(\c{\theta}_{(a,b)})| \c{S}_\Delta(\c{\theta}_{(a,b)})\ge b)&\le (b+C_7)^2\label{A4.6}.
\end{align}
\end{lemma}
The bound \eqref{A4.5} follows from \eqref{A4.6}. We will prove \eqref{A4.6} in Appendix \ref{sec:overshoot}.

Given \eqref{4.5}-\eqref{A4.6}, the proof
of the next lemma is very similar to the one of \cite[Lemma 3]{tTSAW}.

\begin{lemma}\label{gr3}
  \begin{enumerate}[label=(\alph*)]
  \item  There exists a constant $C_8>0$ such that for all $b\ge 2$, $n_1\in \Z_+$, and  $m\in(0,b)\cap \Z$ \[P_{n_1,m}(\c{S}_\Delta(\c{\theta}_{(0,b)}) = 0 )>\frac{b-m}{b+C_8}.\]
  \item  There exists a constant $C_9>0$ such that for all $b\ge 2$, $n_1\in \Z_+$, and $m\in[0,b)\cap\Z_+$ \[E_{n_1,m}[\c{\theta}_{[0,b)}]\le C_9b^3.\]
  \end{enumerate}
\end{lemma}

\begin{proof}
  Note that if $n_1<b$ then $\c{S}_\Delta$ follows a simple symmetric
  random walk for which both statements hold by the standard Gambler's
  ruin estimates. Thus, we shall assume that $n_1\ge b$. Without
  loss of generality we also suppose that $b>a_1+1+C_7$, where $a_1$ is
  defined just above \eqref{A4.3} and $C_7$ is from \eqref{A4.5}.
  
  (a) Let $a_1<m<b$ and $f(s)=-s+C_6e^{-C_5s}$. Note that the process $f(\c{S}_\Delta(j\wedge \c{\theta}_{(a_1,b)})),\ j\ge 0$, is a submartingale relative to filtration  ${\cal G}_j:=\sigma(\t{S}_1(\ell),\c{S}_\Delta(\ell),\chi(\ell),\  0\le\ell\le j)$, $j\ge 0$. Indeed,  either $\t{S}_1$ stays above $b$ and, thus, above $\c{S}_\Delta$ up to time $\c{\theta}_{(a_1,b)}$, and the conditions of Lemma~\ref{subm} are satisfied, or, if prior to $\c{\theta}_{(a_1,b)}$ the process $\t{S}_1$ goes below $b$, $\c{S}_\Delta$ switches to a simple symmetric random walk. By the optional stopping theorem and monotonicity of the function $f$,
  \begin{align*}
    b-m&<f(m)-f(b)\le E_{n_1,m}(f(\c{S}_\Delta(\c{\theta}_{(a_1,b)}))-f(b))\\ &\le E_{n_1,m}(f(\c{S}_\Delta(\c{\theta}_{(a_1,b)}))-f(b)\mid \c{S}_\Delta(\c{\theta}_{(a_1,b)})\le a_1)P_{n_1,m}(\c{S}_\Delta(\c{\theta}_{(a_1,b)})\le a_1)\\ &< \left(b+C_6-E_{n_1,m}(\c{S}_\Delta(\c{\theta}_{(a_1,b)})\mid \c{S}_\Delta(\c{\theta}_{(a_1,b)})\le a_1)\right)P_{n_1,m}(\c{S}_\Delta(\c{\theta}_{(a_1,b)})\le a_1)
  \end{align*}
Since $\c{S}_\Delta(\c{\theta}_{(a_1,b)})\ge 0$, we get that for $a_1<m<b$
\begin{equation}
  \label{lb1}
 P_{n_1,m}(\c{S}_\Delta(\c{\theta}_{(a_1,b)})\le a_1)>\frac{b-m}{b+C_6}.
\end{equation}
Note that, by the remark made at the start of the proof, the above inequality holds for all $(n_1,m)\in\{(n_1,m)\in \Z_+^2\mid n_1<b\ \text{or}\ (n_1\ge b\ \text{and}\ a_1<m<b)\}$

Let us consider the case $0<m\le a_1$. Using \eqref{lb1} we get
\begin{multline*}
 P_{n_1,m}(\c{S}_\Delta(\c{\theta}_{(0,b)}) = 0)\\ \ge P_{n_1,m}(\c{S}_\Delta(\c{\theta}_{(0,a_1]}) = 0)+\sum_{m\ge 0}P_{n_1,m}\left(\c{S}_\Delta(\c{\theta}_{(0,a_1]})> a_1, \t{S}_1(\c{\theta}_{(0,a_1]})=m\right)\\ \times\sum_{a_1<j<b}P_{n_1,m}\left(\c{S}_\Delta(\c{\theta}_{(0,a_1]})=j\mid \c{S}_\Delta(\c{\theta}_{(0,a_1]})> a_1, \t{S}_1(\c{\theta}_{(0,a_1]})=m\right)\frac{b-j}{b+C_6}\\ \times\inf_{0<m'\le a_1,n_1'\in\Z_+} P_{n_1',m'}(\c{S}_\Delta(\c{\theta}_{(0,b)}) = 0)
\end{multline*}
Rearranging the right hand side and using \eqref{A4.5} we see that
\begin{multline*}
  P_{n_1,m}(\c{S}_\Delta(\c{\theta}_{(0,b)}) = 0)\ge P_{n_1,m}(\c{S}_\Delta(\c{\theta}_{(0,a_1]}) = 0)\\+ \frac{b-a_1-1-C_7}{b+C_6}P_{n_1,m}\left(\c{S}_\Delta(\c{\theta}_{(0,a_1]})> a_1\right)\inf_{0<m'\le a_1,n_1'\in\Z_+} P_{n_1',m'}(\c{S}_\Delta(\c{\theta}_{(0,b)}) = 0)\\=\left(1-\frac{b-a_1-1-C_7}{b+C_6}\inf_{0<m'\le a_1,n_1'\in\Z_+} P_{n_1',m'}(\c{S}_\Delta(\c{\theta}_{(0,b)}) = 0)\right)P_{n_1,m}(\c{S}_\Delta(\c{\theta}_{(0,a_1]}) = 0)\\+\frac{b-a_1-1-C_7}{b+C_6}\inf_{0<m'\le a_1,n_1'\in\Z_+}P_{n_1',m'}(\c{S}_\Delta(\c{\theta}_{(0,b)}) = 0).
\end{multline*}
We note that for all $m\le a_1$ and $n_1\in\Z_+$
\begin{align*}
P_{n_1,m}(\c{S}_\Delta(\c{\theta}_{(0,a_1]}) = 0)&\ge 
P_{n_1,m}(\c{S}_\Delta(1) = 0)\\ & \ge P_{n_1,a_1}(\c{S}_\Delta(1)=0\mid \c{S}_1(1)\ge n_1)P_{n_1}(\c{S_1}(1)\ge n_1)\\ &\ge \left(1+\lambda^{-2a_1-1}\right)^{-a_1}P_{n_1}(\c{S_1}(1)\ge n_1).   
\end{align*}
Since for all large $n_1$ the distribution of $\c{S}_1(1) - \c{S}_1(0)$ is close to $\pi$, we get that
\[\inf_{0<m'\le a_1,n_1'\in\Z_+}
P_{n_1',m'}(\c{S}_\Delta(\c{\theta}_{(0,a_1]}) = 0)=:C_{10}>0,\] and 
\begin{equation}
  \label{lb2}
  \inf_{0<m'\le a_1,n_1'\in\Z_+} P_{n_1',m'}(\c{S}_\Delta(\c{\theta}_{(0,b)}) = 0)\ge \frac{C_{10}(b+C_6)}{C_{10}b+C_6+(1-C_{10})(a_1+1+C_7)}.
\end{equation}
Inequalities \eqref{lb1} and \eqref{lb2} imply part (a).

(b) We shall first show that there is a constant $C_{11}>0$ such that for all $n_1\in\Z_+$ and integers $a$ and $m$ satisfying $0\le a<m<b$,
\begin{equation}
  \label{lb3}
  P_{n_1,m}(\c{S}_\Delta(\c{\theta}_{(a,b)})\ge b)>C_{11}\frac{m-a}{b-a}.
\end{equation}
Let $a_2=a_1\vee C_6$, $g(s)=s+C_6e^{-C_5s}$, and assume that $a_2\le a<b$ and $\c{S}_\Delta(0)=m\in(a,b)$. Consider the submartingale $g(\c{S}_\Delta(j\wedge\c{\theta}_{(a,b)})),\ j\ge 0$. Note that, by convexity of $g$ and the fact that $g(0)=C_6<a<g(a)$, if  $0\le\c{S}_\Delta(j)<a$ then $g(\c{S}_\Delta(j))<g(a)$.  Using this observation and \eqref{A4.5} we get
\begin{equation*}
  g(m)-g(a)\le E_{n_1,m}\left(g(\c{S}_\Delta(\c{\theta}_{(a,b)})-g(a)\right)\le (b+C_7+C_6-a) P_{n_1,m}(\c{S}_\Delta(\c{\theta}_{(a,b)})\ge b)
\end{equation*}
and conclude that \eqref{lb3} holds with some constant $C_{12}>0$ and
$a_2\le a<m<b$. Replacing the constant $C_{12}$ with a smaller
one,
\[C_{11}:=\min_{0\le a'<m'\le a_2}\
  \inf_{n_1'\in\Z_+}P_{n_1',m'}(\c{S}_\Delta(\c{\theta}_{(a',a_2]})>a_2)\,C_{12}>0,\]
we obtain \eqref{lb3} for all $0\le a<m<b$ and $n_1\in\Z_+$. A
  lower bound for the infimum in the previous line is obtained
  similarly to the bound just above \eqref{lb2}. Since in this case the exit 
  is through the upper boundary point, we shall restrict to the event
  $\c{S}_1(1)\le n_1'$ instead of $\c{S}_1(1)\ge n_1'$.

Non-trapping, \eqref{4.5}, and the
properties of $\pi_2$ imply that there is a constant $\c{\sigma}>0$
such that
\begin{equation*}
  \label{A4.18}
  2\c{\sigma}^2\le \inf_{n_1,m\in\Z_+}E_{n_1,m}\left[(\c{S}_\Delta(1)-m)^2\right].
\end{equation*}
Now choose $a_3\in\Z_+$ such that for all  $m\ge a_3$,
\begin{equation*}
  \label{A4.19}
  2mC_4e^{-C_5m}<\c{\sigma}^2.
\end{equation*}
Then \eqref{A4.1} ensures that for all $m>a_3$ and $n_1\in\Z_+$,
\[E_{n_1,m}\left[\c{S}_\Delta^2(1)-\c{\sigma}^2\right]\ge E_{n_1,m}\left[(\c{S}_\Delta(1)-m)^2-2m|\c{S}_\Delta(1)-m|+m^2-\c{\sigma}^2\right]\ge m^2 \]  and, hence, $\c{S}_\Delta^2(j)-\c{\sigma}^2j$, with $\c{S}_\Delta(0)>a_3$ and $j=0,1,\ldots,\c{\theta}_{(a_3,b)}$ is a submartingale. Therefore,  \[ E_{n_1,m}\left[\c{S}_\Delta^2(\c{\theta}_{(a_3,b)}\wedge T)-\c{\sigma}^2(\c{\theta}_{(a_3,b)}\wedge T)\right]\ge m^2> 0,\quad\forall T\in\N.\] Using \eqref{A4.6} we get
\begin{equation*}
  \label{A4.23}
E_{n_1,m}\left[\c{\theta}_{(a_3,b)}\right]=\lim_{T\to\infty}E_{n_1,m}\left[\c{\theta}_{(a_3,b)}\wedge T\right]\le\frac{1}{\c{\sigma}^2}(b+C_7)^2.
\end{equation*}
Representing the exit time from $[0,b)$ recursively as exits from $[0,a_3]$ and from $(a_3,b)$ and using $C_{11}/b$ as a lower bound on the right hand side of \eqref{lb3} we get that 
\begin{align*}
  E_{n_1,m}\left[\c{\theta}_{[0,b)}\right]&\le C_{11}^{-1}b\left(\sup_{0\le m\le a_3,
      n_1\in\Z_+}E_{n_1,m}\left[\c{\theta}_{[0,a_3]}\right]+\sup_{a_3< m<b,
      n_1\in\Z_+}E_{n_1,m}\left[\c{\theta}_{(a_3,b)}\right]\right)\\ &\le
  C_{11}^{-1}b\left(\sup_{0\le m\le a_3,
      n_1\in\Z_+}E_{n_1,m}\left[\c{\theta}_{[0,a_3]}\right]+\frac{1}{\c{\sigma}^2}(b+C_7)^2\right).
\end{align*}
To complete the proof of (b) we only need to show that the supremum in
the last expression is finite. The finiteness of the supremum follows
by comparison with a geometric random variable from the inequality
\begin{equation}
  \label{lbe}
  \inf_{0\le m\le a_3,
      n_1\in\Z_+}P_{n_1,m}\left(\c{\theta}_{[0,a_3]}\le (a_3+1)\right)>0,
\end{equation}
Markov property of $(\t{S}_1,\t{S}_2)$, and the definition of
$\c{S}_\Delta$.  To show \eqref{lbe} we note that, when
$\t{S}_1$ is large, its next step distribution is close to
$\pi$. Thus, given any fixed $\ell\in\N$, there is an $n_2\in\N$ such
that the probability that after the next step $S_1$ will be within
$\ell$ units from its current position is bounded away from zero
whenever the current position is larger than $n_2$.  For a given
$\ell$, on this event, the probability that $\c{S}_\Delta$ enters
$(a_3,\infty)$ in the same step is bounded away from $0$. If $\t{S}_1$
is below $b$ then  $\c{S}_\Delta$ follows a simple symmetric random walk, so
the probability that it exits $[0,a_3]$ within the next $a_3+1$ units
of time is bounded away from $0$. There are finitely many integer
values between $b$ and $n_2$, hence, due to non-trapping, if $\t{S}_1$
is in $[b,n_2]$ then the probability that in the next step
$\c{S}_\Delta$ enters $(a_3,\infty)$ is again bounded away from 0. We conclude that \eqref{lbe} holds. 
\end{proof}

We return to the original processes $S_1,S_2,S_\Delta=S_2-S_1$. In line with our previous notation, for an
interval ${\cal I}\subset [0,\infty)$, we let
$\theta_{{\cal I}}=\inf\{\ell\ge 0: S_\Delta(\ell)\not\in {\cal
  I}\}$. Note that $\theta_{(0,\infty)}$ is simply the coalescence
time of $S_1$ and $S_2$. 

\begin{corollary}\label{close} 
    \[\liminf_{b\to\infty}\inf_{\substack{n_1\ge b^4\\0< m\le b}}P_{n_1,m}(\theta_{(0,\infty)}<b^7)=1.\]
  \end{corollary}

  \begin{proof} Let $n_1\ge b^4$ and $0< m\le b$. Note that
    $\{S_{1*}(b^7)\ge b^4/2\}\subset\{\t{S}_{1*}(b^7)\ge
    b^4/2\}$ and 
    \begin{align*}
      P_{n_1,m} &(\theta_{(0,\infty)}\ge b^7)\\
      &= P_{n_1,m}(\theta_{(0,\infty)}\ge b^7,\theta_{(0,\infty)}<\theta_{[0,b^2)})+P_{n_1,m}(\theta_{(0,\infty)}\ge b^7,\theta_{(0,\infty)}>\theta_{[0,b^2)})\\  
      &\le P_{n_1,m}(\theta_{[0,b^2)}\ge b^7)+P_{n_1,m}(S_\Delta(\theta_{(0,b^2)})\ge b^2)\\ 
      & \le 1-P_{n_1,m}(\theta_{[0,b^2)}< b^7, S_{1*}(b^7)\ge b^4/2, \t{S}_{1*}(b^7)\ge b^4/2)\\ &\qquad\qquad +1-P_{n_1,m}(S_\Delta(\theta_{(0,b^2)}) \leq 0,S_{1*}(b^7)\ge b^4/2,\theta_{[0,b^2)}<b^7, \t{S}_{1*}(b^7)\ge b^4/2)
      \\ & =  1-P_{n_1,m}(\c{\theta}_{[0,b^2)}< b^7, \t{S}_{1*}(b^7)\ge b^4/2,S_{1*}(b^7)\ge b^4/2)\\ &\qquad\qquad +1-P_{n_1,m}(\c{S}_\Delta(\theta_{(0,b^2)})= 0,\t{S}_{1*}(b^7)\ge b^4/2,\c{\theta}_{[0,b^2)}<b^7,S_{1*}(b^7)\ge b^4/2)
      \\ & \le 1-P_{n_1,m}(\c{\theta}_{[0,b^2)}< b^7)+P_{n_1,m}(S_{1*}(b^7)< b^4/2)\\ &\qquad\ \  +1-P_{n_1,m}(\c{S}_\Delta(\theta_{(0,b^2)})= 0)+P_{n_1,m}(\c{\theta}_{[0,b^2)}\ge b^7)+P_{n_1,m}(S_{1*}(b^7)< b^4/2)\\ & \le  2P_{n_1,m}(\c{\theta}_{[0,b^2)}\ge b^7)+1-P_{n_1,m}(\c{S}_\Delta(\theta_{(0,b^2)})= 0)+2P_{n_1,m}(S_{1*}(b^7)< b^4/2)\\ & \le                                                                 \frac{2E_{n_1,m}[\c{\theta}_{[0,b^2)}]}{b^7}+1-P_{n_1,m}(\c{S}_\Delta(\c{\theta}_{(0,b^2)}= 0)+2P_{n_1,m}(S_{1*}(b^7)< b^4/2)\\ &\le \frac{2C_9}{b}+1-\frac{b^2-b}{b^2+C_8}+2P_{n_1,m}(S_{1*}(b^7)< b^4/2),
    \end{align*}
    where in the last inequality we used Lemma~\ref{gr3} with $b^2$
    instead of $b$. The stated result follows by applying \eqref{min}
    with $b$ replaced by $b^4$ and $\delta=1/4$ and letting $b$ go to
    infinity.
  \end{proof}

\section{Proof of the overshoot lemma}\label{sec:overshoot}

In this section we prove Lemma~\ref{os} using the
framework detailed in Appendix~\ref{sec:urn}.
Recall that the increments of $\c{S}_\Delta$ are either given by a
simple symmetric random walk or by the increments of the difference of
two directed edge local time processes, the latter of which can be
related to urn processes as in Lemma \ref{lem:blpmulti-tp}. Therefore,
it is enough to show that there are constants $C,c>0$ such that for
all $b,n,m,x \in \N$
\begin{align}\label{Dbndiff-exptail}
 P_i\left( \mathfrak{D}_{\beta_{n+m}}-\mathfrak{D}_{\beta_n} \geq b+x \mid \mathfrak{D}_{\beta_{n+m}}-\mathfrak{D}_{\beta_n} \geq b \right) &\leq C e^{-cx}, \quad i \in \{0,1\},\ \ \text{and} 
\\\label{tDbndiff-exptail}
 P_0\left( \widetilde{\mathfrak{D}}_{\tilde\beta_{n+m}}-\widetilde{\mathfrak{D}}_{\tilde\beta_n} \geq b+x \mid \widetilde{\mathfrak{D}}_{\tilde\beta_{n+m}}-\widetilde{\mathfrak{D}}_{\tilde\beta_n} \geq b \right) &\leq C e^{-cx}. 
\end{align}
We will only prove \eqref{Dbndiff-exptail} below as the proof of
\eqref{tDbndiff-exptail} is similar.  A key tool we will use is the
following lemma which gives exponential decay on both tails of the
Markov chain $(\mathfrak{D}_{\beta_n})_{n\ge 0}$.
\begin{lemma}\label{lem:Dtn-tails}
 There exists constants $C,c>0$ such that 
 \begin{align}
  &P_x\left( \mathfrak{D}_{\beta_n} = y+\ell \right) \leq C e^{-c \ell} P_x\left( \mathfrak{D}_{\beta_n} = y \right), \quad \forall\, y\ge 0,\,  x\leq y, \, \ell \geq 0, \, n\geq 1, \label{uniftails-r} \\
  &P_x\left( \mathfrak{D}_{\beta_n} = y-\ell \right) \leq C e^{-c \ell} P_x\left( \mathfrak{D}_{\beta_n} = y \right), \quad \forall\, y\le 0,\,  x\in\Z,\, \ell \geq 0,\, n\geq 1.  \label{uniftails-l}
 \end{align}
\end{lemma}
\begin{remark}
  A much stronger right tail bound for $x\in \{0,1\}$ is given in
  \cite[(3.29),(3.30)]{tTSAW}. Since we need the result starting from any
  $x\le y$, we included a proof. The proof is similar to the one in
  \cite{tTSAW}. It yields a stronger bound than \eqref{uniftails-r} or
  even than \cite[(3.29)]{tTSAW}. Since we do not need this stronger result, we opted
  for a symmetric statement for both right and left tails.
\end{remark}

\begin{proof}
 We first prove  \eqref{uniftails-l}. First of all, note that a simple computation yields that
 \begin{equation}\label{d-ud}
  \frac{P_x\left( \mathfrak{D}_1=x-1 \right)}{P_x\left( \mathfrak{D}_1=x+1, \, \mathfrak{D}_2=x \right)} 
 = \frac{1+\l^{2x+1}}{\l^{2x-1}} 
 = \l^{-2x+1} + \l^2, \quad \forall x \in \Z. 
 \end{equation}
 Using this, we get that for any $x\in \Z$, $y\le 0$, and
 $n\geq 1$,
 \begin{align}
  P_x\left( \mathfrak{D}_{\beta_n} = y-1 \right) 
   &=P_x\left(\mathfrak{D}_{\beta_n}=y-1,\,\mathfrak{D}_{\beta_n-1} = y\right) \nonumber \\ & =\sum_i P_x\left(\mathfrak{D}_{i+1}=y-1,\,\mathfrak{D}_i = y,\,\beta_{n-1}\leq i < \beta_n\right) \nonumber \\
   &= \sum_i P_{y}\left( \mathfrak{D}_1 = y-1 \right)P_x\left(  \mathfrak{D}_i=y,\,\beta_{n-1}\leq i < \beta_n \right) \nonumber \\
  &= \left( \l^{-2y+1} + \l^2 \right) \sum_i P_{y}\left( \mathfrak{D}_1=y+1, \mathfrak{D}_2 = y \right) P_x\left(\mathfrak{D}_i=y,\, \beta_{n-1}\leq i < \beta_n \right)  \nonumber \\
  &\leq \left( \l+ \l^2 \right) P_x\left( \mathfrak{D}_{\beta_n} = y \right). \label{lt-1}
 \end{align}
 Since the multiplicative factor $\l+\l^2$ is strictly less than 1
for all $y$ sufficiently negative, the inequality
\eqref{uniftails-l} follows easily by iterating the bound in
\eqref{lt-1}.
 
 For the proof of  \eqref{uniftails-r}, we let
 \[
  r_n(y) = \sup_{x\leq y} \frac{P_x\left( \mathfrak{D}_{\beta_n} = y+1 \right)}{P_x\left( \mathfrak{D}_{\beta_n} = y \right)}, \qquad y \in \Z,\ n\geq 1. 
 \]
To bound $r_1(y)$, note that \eqref{d-ud} implies
\begin{align}
 P_x(\mathfrak{D}_{\beta_1} = y+1) 
 &= P_x\left( \mathfrak{D}_{y-x+1} = y+1 \right) P_{y+1}\left( \mathfrak{D}_1 = y+2, \, \mathfrak{D}_2 = y+1 \right)\nonumber \\
 &= P_x\left( \mathfrak{D}_{y-x+1} = y+1 \right) P_{y+1}\left( \mathfrak{D}_1 = y \right) \frac{\l^{2y+1}}{1+\l^{2y+3}} \nonumber \\
 &= P_x(\mathfrak{D}_{\beta_1} = y) \frac{\l^{2y+1}}{1+\l^{2y+3}}\ \ \Rightarrow\ \ r_1(y) = \frac{\l^{2y+1}}{1+\l^{2y+3}} \label{r1}. 
\end{align}
Next, we will use \eqref{r1} and an induction argument to get
bounds on $r_n(y)$ for $n\geq 2$.  For $x\leq y$ and $n\geq 2$ note
that
\begin{align*}
 P_x\left( \mathfrak{D}_{\beta_n} = y+1 \right) 
 &= P_x\left( \mathfrak{D}_{\beta_n} = y+1,\, \mathfrak{D}_{\beta_{n-1}} \leq y+1 \right) + P_x\left( \mathfrak{D}_{\beta_n} = y+1,\, \mathfrak{D}_{\beta_{n-1}} = y+2 \right) \\
 &= \sum_i P_x\left( \beta_{n-1}\leq i < \beta_n, \, \mathfrak{D}_i = y+1 \right) P_{y+1}\left(\mathfrak{D}_1 = y+2, \, \mathfrak{D}_2 = y+1 \right) \\
 &\qquad + P_x\left( \mathfrak{D}_{\beta_{n-1}} = y+2 \right)P_{y+2}\left( \mathfrak{D}_1 = y+1 \right) \\
  &\leq \sum_i P_x\left( \beta_{n-1}\leq i < \beta_n, \, \mathfrak{D}_i = y+1 \right) P_{y+1}\left(\mathfrak{D}_1 = y \right) \frac{\l^{2y+1}}{1+\l^{2y+3}}  \\
 &\qquad + r_{n-1}(y+1)P_x\left( \mathfrak{D}_{\beta_{n-1}} = y+1 \right)P_{y+1}\left( \mathfrak{D}_1 = y \right) \frac{1+\l^{2y+1}}{1+\l^{2y+3}} \\
 &\leq  \left( \frac{\l^{2y+1}}{1+\l^{2y+3}}  + \frac{1+\l^{2y+1}}{1+\l^{2y+3}} r_{n-1}(y+1) \right) P_x\left( \mathfrak{D}_{\beta_n} = y \right).
\end{align*}
where the second to last inequality follows from the definition of $r_{n-1}(y+1)$ and the fact that $\frac{P_{y+2}\left( \mathfrak{D}_1 = y+1 \right)}{P_{y+1}\left( \mathfrak{D}_1 = y \right)} = \frac{1+\l^{2y+1}}{1+\l^{2y+3}}$. 
Thus, we have shown that 
\[
 r_n(y) \leq \frac{\l^{2y+1}}{1+\l^{2y+3}}  + \frac{1+\l^{2y+1}}{1+\l^{2y+3}} r_{n-1}(y+1), \quad\forall n\geq 1, \, y\in \Z. 
\]
From this and \eqref{r1} we get the estimate
\begin{equation*}
 r_n(y) \leq  \sum_{k=1}^n \frac{(1+\l^{2y+1})\l^{2(y+k)-1}}{(1+\l^{2(y+k)-1})(1+\l^{2(y+k)+1})}
 \leq \frac{(1+\l^{2y+1})\l^{2y+1}}{1-\l^2}.
\end{equation*}
Therefore, $r_n(y)\le \frac{\l^{2y+1}}{1-\l}$ for all $y\geq 0$ and $n\in \N$, and the claim in \eqref{uniftails-r} follows.  
\end{proof}

Returning to the proof of \eqref{Dbndiff-exptail}, we claim that
there exist constants $C,c>0$ such that uniformly over
$b,n,m,x \in \N$ we have the following two inequalities:
\begin{align}
 P_i\left( \mathfrak{D}_{\beta_{n+m}}-\mathfrak{D}_{\beta_n} \geq b+x, \, \mathfrak{D}_{\beta_n} \geq -b \right)
 &\leq C e^{-cx} P_i\left( \mathfrak{D}_{\beta_{n+m}}-\mathfrak{D}_{\beta_n} \geq b, \, \mathfrak{D}_{\beta_n} \geq -b \right); \label{Dbn1} \\
 P_i\left( \mathfrak{D}_{\beta_{n+m}}-\mathfrak{D}_{\beta_n} \geq b+x, \, \mathfrak{D}_{\beta_n} < -b \right)
 &\leq C e^{-cx} P_i\left( \mathfrak{D}_{\beta_{n+m}}-\mathfrak{D}_{\beta_n} \geq b, \, \mathfrak{D}_{\beta_n} < -b \right). \label{Dbn2}
\end{align}
To prove \eqref{Dbn1}, note that 
\begin{align*}
 P_i\left( \mathfrak{D}_{\beta_{n+m}}-\mathfrak{D}_{\beta_n} \geq b+x, \, \mathfrak{D}_{\beta_n} \geq -b \right)
 &= \sum_{\ell\geq 0} P_i\left( \mathfrak{D}_{\beta_n} = -b+\ell \right) P_{-b+\ell} \left( \mathfrak{D}_{\beta_m} \geq \ell+x \right) 
 \\
 &\leq C e^{-cx} \sum_{\ell\geq 0} P_i\left( \mathfrak{D}_{\beta_n} = -b+\ell \right) P_{-b+\ell} \left( \mathfrak{D}_{\beta_m} \geq \ell \right) \\
 &= C e^{-cx} P_i\left( \mathfrak{D}_{\beta_{n+m}}-\mathfrak{D}_{\beta_n} \geq b, \, \mathfrak{D}_{\beta_n} \geq -b \right), 
\end{align*}
where the inequality follows from \eqref{uniftails-r}. 

The proof of \eqref{Dbn2} is similar, but slightly more delicate. First of all, note that 
\begin{align}
 &P_i\left( \mathfrak{D}_{\beta_{n+m}}-\mathfrak{D}_{\beta_n} \geq b+x, \, \mathfrak{D}_{\beta_n} < -b \right) \nonumber \\
 &\leq P_i\left( \mathfrak{D}_{\beta_n} < -b -\fl{x/2} \right) + P_i\left( \mathfrak{D}_{\beta_{n+m}}-\mathfrak{D}_{\beta_n} \geq b+x, \, \mathfrak{D}_{\beta_n} \in \left[ -b -\fl{x/2}, -b \right) \right)  \nonumber \\
 &\leq C e^{-c\fl{x/2}} P_i\left( \mathfrak{D}_{\beta_n} < -b \right) + \sum_{\ell=1}^{\fl{x/2}} P_i\left( \mathfrak{D}_{\beta_n} = -b-\ell \right) P_{-b-\ell}\left( \mathfrak{D}_{\beta_m} \geq x-\ell \right), \label{Dbn2-a}  
\end{align}
where we used \eqref{uniftails-l} for the last inequality. 
For the second term in \eqref{Dbn2-a}, note that if $\ell \leq \fl{x/2}$ then we can apply \eqref{uniftails-r} to get that 
\[
 P_{-b-\ell}\left( \mathfrak{D}_{\beta_m} \geq x-\ell \right)
 \leq C e^{-c(x-\fl{x/2})} P_{-b-\ell}\left( \mathfrak{D}_{\beta_m} \geq \fl{x/2} - \ell \right) 
 \leq C e^{-cx/2}. 
\]
Applying this to \eqref{Dbn2-a} we get 
\begin{align}
 P_i\left( \mathfrak{D}_{\beta_{n+m}}-\mathfrak{D}_{\beta_n} \geq b+x, \, \mathfrak{D}_{\beta_n} < -b \right) 
 &\leq C e^{-c\fl{x/2}} P_i\left( \mathfrak{D}_{\beta_n} < -b \right) + C e^{-cx/2} P_i\left(\mathfrak{D}_{\beta_n} < -b \right) \nonumber \\
 &\leq C' e^{-cx/2}P_i\left(\mathfrak{D}_{\beta_n} < -b \right). \label{Dbn2-b}
\end{align}
The proof of \eqref{Dbn2} will then be finished if we can bound $P_i(\mathfrak{D}_{\beta_n} <-b )$ in the last term with a multiple of
$P_i\left(  \mathfrak{D}_{\beta_n} < -b, \, \mathfrak{D}_{\beta_{n+m}} - \mathfrak{D}_{\beta_n} \geq b \right)$.
To this end, since 
\[P_i\left( \mathfrak{D}_{\beta_n}< -b, \, \mathfrak{D}_{\beta_{n+m}} - \mathfrak{D}_{\beta_n} \geq b \right)
= \sum_{k<-b} P_i( \mathfrak{D}_{\beta_n} = k ) P_k\left( \mathfrak{D}_{\beta_m} \geq k+b \right),\] 
it is enough to prove that 
\begin{equation}\label{Dmoveb}
 \inf_{b,m\in \N, \, k < -b} P_k\left( \mathfrak{D}_{\beta_m} \geq k+b \right) > 0. 
\end{equation}
To prove \eqref{Dmoveb}, note that 
 \begin{multline*}
  P_k\left( \mathfrak{D}_{\beta_m} \geq k+b \right)
  \geq P_k\left( \mathfrak{D}_b = k+b \right) P_{k+b}\left( \mathfrak{D}_{\beta_m} \geq k+b \right) \\
  = \left( \prod_{x=k}^{k+b-1} \frac{\l^{2x-1}}{1+\l^{2x-1}} \right)P_{k+b}\left( \mathfrak{D}_{\beta_m} \geq k+b \right) 
  \geq \left( \prod_{x\leq -2} \frac{\l^{2x-1}}{1+\l^{2x-1}} \right) \left( \inf_{\ell \leq -1} P_{\ell}\left( \mathfrak{D}_{\beta_m} \geq \ell \right) \right).
 \end{multline*}
The infinite product in the last line is strictly positive, so we need only to show that the infimum of probabilities are also bounded away from zero, uniformly in $m$. 
To this end, since the convergence in distribution of $\mathfrak{D}_{\beta_m}$ as $m\to \infty$ implies that 
$\inf_{m\geq 1} P_{-1}(\mathfrak{D}_{\beta_m} \geq -1 ) > 0$, then it's enough to show that the probabilities $P_{\ell}\left( \mathfrak{D}_{\beta_m} \geq \ell \right)$  are non-increasing in $\ell$  so that the infimum is achieved at $\ell=-1$. The claimed monotonicity of $P_{\ell}\left( \mathfrak{D}_{\beta_m} \geq \ell \right)$ in $\ell$ is readily seen from the classical Rubin's construction of a generalized Polya's urn process using independent exponential random variables, see \cite[pp.\,226--227]{Dav90} for the details of this construction. 
Indeed, the event $\mathfrak{D}_{\beta_m} \geq \ell = \mathfrak{D}_0$ occurs if the corresponding urn process has $m$ ``red'' draws before $m$ ``blue'' draws. 
Therefore, Rubin's construction gives that if $\{\gamma_i\}_{i\leq m}$ and $\{\gamma'_i\}_{i\geq m}$ are independent Exp(1) random variables then  
\[
 P_{\ell}\left( \mathfrak{D}_{\beta_m} \geq \ell \right) = P\left( \sum_{i=1}^m \frac{\gamma_i}{\l^{2(\ell+i-1)-1}} < \sum_{i=1}^m \frac{\gamma'_i}{\l^{2(i-1)}}  \right), 
\]
and the right side is clearly non-increasing in $\ell$. 
This completes the proof of \eqref{Dmoveb}, and thus also of \eqref{Dbn2}. 

\section*{Acknowledgments}
The authors would like to thank Thomas Mountford for many helpful
  discussions.  This work was partially supported by Collaboration
  Grants for Mathematicians \#523625 (E.K.) and \#635064 (J.P.) from
  the Simons Foundation. A part of this work was done during E.K.'s
  visiting appointment at NYU Shanghai in 2024-2025. E.K.\ also thanks
  the Simons Laufer Mathematical Sciences Institute for support
  through the membership in the program ``Probability and Statistics
  of Discrete Structures'' and for a stimulating research
  environment. 


\end{document}